\documentclass[acmsmall,screen,nonacm]{acmart}

\usepackage{amsmath}
\usepackage{amsfonts} 
\usepackage{amsthm}

\usepackage{dsfont}
\usepackage{stmaryrd}

\usepackage{tikz}
\usetikzlibrary{arrows,automata}

\usepackage{adjustbox}

\usepackage{xcolor}
\usepackage{pifont}

\usepackage{graphicx}
\setkeys{Gin}{width=\linewidth,totalheight=\textheight,keepaspectratio}
\graphicspath{{figures/}}
\usepackage{subcaption}
\usepackage{wrapfig}

\usepackage{booktabs}

\usepackage{units}

\usepackage{fancyvrb}
\fvset{fontsize=\normalsize}

\usepackage{multicol}

\usepackage{enumitem} 

\DeclareFixedFont{\ttb}{T1}{txtt}{bx}{n}{7} 
\DeclareFixedFont{\ttm}{T1}{txtt}{m}{n}{7}  

\definecolor{deepblue}{rgb}{0,0,0.5}
\definecolor{deepred}{rgb}{0.6,0,0}
\definecolor{deepgreen}{rgb}{0,0.5,0}

\usepackage{listings}

\newcommand\pythonstyle{\lstset{
frameround=tttt,
language=python,
basicstyle=\small\normalfont\sffamily,
morekeywords={_lambda, nu, mu, buffer_size, G},              
keywordstyle=\ttb\color{deepblue},
emph={},          
emphstyle=\ttb\color{deepgreen},
emph={[2]v, t, u, w, s, T, X, X_equi},    
emphstyle={[2]\ttb\color{deepred}},
stringstyle=\color{deepgreen},
frame=single,                         
showstringspaces=false
}}

\lstnewenvironment{python}[1][]
{
\pythonstyle
\lstset{#1}
}
{}

\AtBeginDocument{%
  \providecommand\BibTeX{{%
    \normalfont B\kern-0.5em{\scshape i\kern-0.25em b}\kern-0.8em\TeX}}}

\newcommand{\barF}{\bar{F}}
\newcommand{\E}{\mathbb{E}}
\newcommand\esp[1]{\E[#1]}
\newcommand{\Lf}{L_{\text{fast}}}

\newcommand{\calX}{\mathcal{X}}
\newcommand{\calY}{\mathcal{Y}}
\newcommand{\calT}{\mathcal{T}}
\newcommand{\calD}{\mathcal{D}}
\newcommand{\calC}{\mathcal{C}}
\newcommand{\R}{\mathbb{R}}
\newcommand{\abs}[1]{\left| #1 \right|}
\newcommand{\norm}[1]{\left\| #1 \right\|}

\newcommand\toN{^{(N)}}
\newcommand{\bX}{\boldsymbol{X}}
\newcommand{\bx}{\boldsymbol{x}}
\newcommand{\bY}{\boldsymbol{Y}}
\newcommand{\by}{\boldsymbol{y}}

\newcommand{\Hshort}[1]{D_x(h\circ \phi_{t-#1})(X_{#1}) \Gfast_F(X_{#1}, Y_{#1}) }
\newcommand{\Xinf}{\bX_{\infty}}
\newcommand{\Yinf}{\bY_{\infty}}
\newcommand{\Phinf}{\phi_{\text{\tiny$\infty$}}}
\newcommand{\Gslow}{G^{\text{slow}}}
\newcommand{\Gfast}{G^{\text{fast}}}

\newtheorem{theorem}{Theorem}
\newtheorem{lemma}[theorem]{Lemma}
\newtheorem{definition}[theorem]{Definition}
\newtheorem{note}[theorem]{Note}
\newtheorem{proposition}[theorem]{Proposition}

\AtBeginDocument{%
  \providecommand\BibTeX{{%
    Bib\TeX}}}

\setcopyright{acmlicensed}
\acmJournal{POMACS}
\acmYear{2023} \acmVolume{7} \acmNumber{1} \acmMonth{3} \acmPrice{15.00}\acmDOI{10.1145/3579336}

\begin{document}

\title{Bias and Refinement of Multiscale Mean Field Models}

\author{Sebastian Allmeier}
\email{sebastian.allmeier@inria.fr}
\orcid{0000-0003-4629-6348}
\author{Nicolas Gast}
\orcid{0000-0001-6884-8698}
\email{nicolas.gast@inria.fr}
\affiliation{%
  \institution{Inria}
  \city{Grenoble}
  \country{France}
}
\affiliation{
  \institution{Univ. Grenoble Alpes}
  \city{Grenoble}
  \country{France}
}

\renewcommand{\shortauthors}{Allmeier and Gast}

\begin{abstract}
    Mean field approximation is a powerful technique which has been used in many settings to study large-scale stochastic systems. In the case of two-timescale systems, the approximation is obtained by a combination of scaling arguments and the use of the averaging principle. This paper analyzes the approximation error of this `average' mean field model for a two-timescale model $(\boldsymbol{X}, \boldsymbol{Y})$, where the slow component $\boldsymbol{X}$ describes a population of interacting particles which is fully coupled with a rapidly changing environment $\boldsymbol{Y}$. The model is parametrized by a scaling factor $N$, e.g. the population size, which as $N$ gets large decreases the jump size of the slow component in contrast to the unchanged dynamics of the fast component. 
    We show that under relatively mild conditions, the `average' mean field approximation has a bias of order $O(1/N)$ compared to $\mathbb{E}[\boldsymbol{X}]$.
    This holds true under any continuous performance metric in the transient regime, as well as for the steady-state if the model is exponentially stable. To go one step further, we derive a bias correction term for the steady-state, from which we define a new approximation called the refined `average' mean field approximation whose bias is of order $O(1/N^2)$. This refined `average' mean field approximation allows computing an accurate approximation even for small scaling factors, i.e., $N\approx 10 -50$. We illustrate the developed framework and accuracy results through an application to a random access CSMA model. 
\end{abstract}


\maketitle


\section{Introduction}

The mean field approximation finds widespread application when interested in analyzing the macroscopic behavior of large-scale stochastic systems composed of interacting particles. Its assets lie in a reduction of the model complexity, simplified analysis of the system due to absence of stochastic components, and reduction of computation time compared to a stochastic simulation. The mean field approximation can even yield closed form solutions for the steady-state, e.g., for the well known JSQ(d) model \cite{mitzenmacher_power_2001}. The mean field approximation is generally given by a set of ordinary differential equations which arise from the assumption that, for large system sizes, the evolution of the particles are stochastically independent of another. This idea works well if the number of particles is large and if the particles can be clustered into a few groups with statistically identical behavior \cite{kurtz_solutions_1970,kurtz_strong_1978,benaim_class_2008,le_boudec_generic_2007,gast_markov_2012}. The framework established by Kurtz \cite{kurtz_solutions_1970} to derive (weak) convergence results for the stochastic system justifying the use of the mean field approximations finds sustained attention in the literature.

More recently, the authors of \cite{gast_expected_2017,ying_approximation_2016} showed that for finite system sizes the bias of the mean field approximation is of order $1/N$ when compared to the mean behavior of the system. Here, $N$ is the scaling parameter, which usually refers to the number of homogeneous particles in the system. Additional works such as \cite{gast_refined_2017,gast_size_2019} introduced corrections, called refinement terms, which effectively increase the rate of accuracy of the approximation and therefore the rate of convergence. The most notable term is the first-order bias refinement, since it offers a convincing trade-off between a significant accuracy gain and additional computation cost. 

While these classical mean field results hold for a broad class of models, most of the results can not be transferred to systems with more intricate dynamics such as the two-timescale case, studied for instance in \cite{benaim_class_2008,bordenave_particle_2009}. A two-timescale process consists of two coupled components, one evolving slowly compared to the other. The slowly evolving component is often represented by a system of interacting particles, where the state of each particle evolves as a function of the empirical distribution of all particles but also as a function of the state of the fast component, e.g., a fast changing environment. These types of processes and their `averaged' mean field adaptation have been of interest since the 1960s and became increasingly relevant in the study of modern and complex systems. We refer to \cite{cdi_springer_books_10_1007_978_0_387_73829_1} for an extensive literature discussion.  An important area of application comes from the field of computer networks. Examples include loss networks \cite{hunt_large_1994}, large-scale random access networks with interference graphs \cite{cecchi_mean-field_2019,castiel_induced_2021} or storage networks \cite{feuillet_scaling_2014}.  Another recent field of literature from which we draw inspiration are chemical reaction networks. The works of Kang et al.~\cite{kang_separation_2013,kang_central_2014} and Ball et al.~\cite{ball_asymptotic_2006} establish central limit theorems for large multiscale models motivated by biological and chemical processes. Another application in the field of biology is given by \cite{robert_stochastic_2021} who use the mean field idea to study neural plasticity models.

\paragraph*{Contributions} The aforementioned papers underline that the mean field idea can be adapted to two-timescale models and prove that the `average' mean field approximation is asymptotically exact as $N$ goes to infinity. They do not, however, provide theoretical results which guarantee the accuracy or performance bounds of the approximation for finite systems sizes. This paper aims at filling this gap.
We derive accuracy bounds for the `average' mean field approximation in the transient regime and steady-state which show that its bias is of order $O(1/N)$, $N$ being the scaling parameter of the stochastic system. We further derive bias correction terms for the steady-state, from which we define a new approximation called the refined `average' mean field approximation, whose bias is of order $O(1/N^2)$. 
To prove these accuracy bounds, we develop a framework for two-timescale stochastic models whose slow component is comparable to the concept of density dependent population processes as introduced by Kurtz \cite{kurtz_strong_1978}. Based on this representation, we utilize a combination of generator comparison techniques, Poisson equations as well as derivative bounds on the solution of the Poisson equation. This allows us to bound the bias with respect to the scaling parameter $N$. To take it a step further, we then prove the existence of correction terms which approximate the bias of the `average' mean field and allow defining the refined `average' mean field. To support the practical application of the refinements,  we provide an algorithmic way to compute the correction terms. This includes methods to numerically solve the Poisson equation and obtain its derivatives. We illustrate the computation of the refinement terms and confirm the accuracy of the obtained bounds by considering a random access CSMA model. Using the example we show that even for relatively small $N\approx10$ the refined `average' mean field approximation almost exactly indicates the steady-state of the model. 

\paragraph{Methodological Advances \& Technical Challenges} To obtain our results, we build on the recent line of work on Stein’s method \cite{stein_approximate_1986}. This method allows calculating the distance between two random variables by looking at the distance between the generators of two related systems. Recently, the method reemerged in publications within the stochastic network community, in particular the works of Braverman et al. \cite{braverman_steins_2017,braverman_steins_2017-1,braverman_prelimit_2021}. In this paper, we use the Poisson equation idea in two ways. First and foremost, we use a Poisson equation, called the `fast' Poisson equation, to bound the distance between a function -- in this case the drift -- and its average version given by the steady-state distribution of the fast process as introduced in Section~\ref{sec:fast_poisson_regularity}. This step is integral to our paper and constitutes the building block to obtain the error bounds and closed form expressions as it allows to analyze the distance between the coupled stochastic process and its decoupled counterpart. The analysis however bears many technical intricacies one needs to overcome. This includes  solving the Poisson equation, stating its derivative bounds and deducing computable expressions used to calculate the bias term.
Second, for our steady-state results, we make use of another Poisson equation to compare the stochastic system to the equilibrium point of the approximation. The latter is related to the methods used in \cite{ying_approximation_2016,gast_refined_2017} but significantly extends the ideas as the derivation of the bias exhibits novel refinement terms which originate from the coupling and correct the error of the `averaging' method used for the mean field. The closely linked technical challenges include obtaining numerically feasible formulas for the `new' refinement terms. Here, we utilize the derived solution of the `fast' Poisson equation to further specify the fluctuations of the stochastic system around the equilibrium point. Carefully analyzing the combination of the two Poisson equations enables us to obtain the new refinement terms. The closed-form bias terms that we obtained from the steady-state analysis are significantly more complicated than the ones of \cite{gast_refined_2017} as they further correct the error made by the averaging method.

\paragraph{Applicability \& Numerical Difficulties} In our paper, we make use of a CSMA model to demonstrate the applicability of our results.
There are several other examples captured by our framework, such as the Michaelis-Menten enzyme model of \cite{kang_central_2014} or the storage network investigated in \cite{feuillet_scaling_2014}. For the latter, the authors observe that when looking at the right timescale the loss of the network can be characterized by a local equilibrium which is obtained using the averaging method. The biochemical Michaelis-Menten enzyme model describes the dynamics between three time-varying species, the enzyme, a substrate and a product. Following the description of the model as in \cite{kang_central_2014} and by using the right scaling arguments, the model exhibits two timescales, the fast reacting and state changing enzymes and the slower changing concentrations of the substrate and product. Our framework can be used on both examples to guarantee accuracy results and can be used to compute refined approximations.

To compute the `average' mean field and refinement terms, one has to overcome numerical difficulties which arise from the averaging method. First, in order to compute the `average' drift, one needs to compute the steady-state probabilities of the fast system, which might not be available in the closed form as for the CSMA model. For this case we provide computational notes in the appendix which aim to facilitate the computation. Another problem for the refinement term is the need for the first and second derivatives of the `average' drift. For the CSMA model we used symbolic representation of the transition rates from which we define the drift and its average version. This method allows to numerically compute the derivatives using `sympy', a Python library for symbolic computation. This method is relatively easy to implement but has a very large computation time. This computation time could be reduced by implementing a faster computation of the derivatives. In fact, to compute the derivative of the matrix $K^+$ which is closely linked to the solution of the `fast' Poisson, we provide supplementary computational notes which describe how to obtain an efficient implementation.

\paragraph*{Roadmap}
The paper is organized as follows. In Section~\ref{sec:model}, we formally introduce the two-timescale model, its corresponding `average' mean field approximation and make regularity assumptions on the system. In Section~\ref{sec:main_results} we state the main results of this paper. Section~\ref{sec:transient} states the results for the transient regime, Section~\ref{sec:steady-state} for the steady-state and Section~\ref{sec:refinement} justifies the existence of bias correction terms. Section~\ref{sec:proofs} holds the proofs of the aforementioned results. In Section~\ref{sec:example_csma_model} we apply our theoretical results to the unsaturated random-access network model of \cite{cecchi_mean-field_2018,cecchi_mean-field_2019}. Some technical lemmas and definitions are postponed to the appendix. 

\paragraph*{Reproducibility}
The code to reproduce the paper along with all figures and the implementation of the unsaturated random-access network model is available at \url{https://gitlab.inria.fr/sallmeie/bias-and-refinement-of-multiscale-mean-field-models}.


\section{Stochastic System and Mean Field Approximation}
\label{sec:model}

We consider a two-timescale, coupled, continuous time Markov chain $(\bX_s\toN,\bY_s\toN)_{s\geq0}$ parametrized by a scaling factor $N$, for which we study the behavior as $N$ tends to infinity. As we will see in the examples, $N$ typically represents the number of objects that interact together. This section introduces the precise model and fixes notations. 

\subsection{Model} \label{sec:def_stoch_sys}

For a fixed scaling parameter $N$, the stochastic process $(\bX\toN_s,\bY\toN_s)_{s\geq0}$ is a continuous time Markov chain that evolves in a state-space $\calX\toN\times\calY$. The set $\calY$ is finite and does not depend on $N$.  We further require that for all $N$, the sets $\calX\toN$ are subsets of a convex and compact set $\calX\subset\R^{d_x}$. In what follows, unless it is ambiguous in the context, we drop the dependence on $N$ to lighten the notations.

This model has two-timescales in the sense that the size order of jumps of $\bY_s$ is $N$-times larger than the ones of $\bX_s$.  More precisely, we assume that there exists a finite number of transitions $(\ell,\by')\in\calT$ with their corresponding transition rate functions $\alpha_{\ell,\by'}\geq 0$, both being independent of $N$, such that for all possible states $(\bX_s,\bY_s)\in\calX\times\calY$:
\begin{align}
    (\bX_s, \bY_s) \text{ jumps to } (\bX_s + \ell/N, \by') && \text{ at rate } N \alpha_{\ell, \by'}(\bX_s,\bY_s). \label{eq:system_def}
\end{align}
The above defines continuous time Markov chains with discrete state-space whose realizations are Càdlàg, i.e., right continuous with a left limit for every time $t$. Note that in Equation~\eqref{eq:system_def}, we assume that the transition rates $\alpha$ are defined for all $x\in\calX$ (and not just for $x\in\calX\toN$).

The notion of slow-fast system comes from the fact that the jumps of the slow component $\calX$ are $O(N)$ times smaller than the jumps of the fast component $\calY$ while transition rates are of the same scale. As we will see later, the different timescales imply that for large $N$, the slow component $\bX_s$ will 'see' the fast component $\bY_s$ as if it is stationary with distribution $(\pi_{\by}(\bX_s))_{\by\in \calY}$ which we formally define in Section~\ref{sec:def_drift}. 

\subsection{Drift, Average Drift and Mean Field Approximation}
\label{sec:def_drift}

The jumps of the stochastic system~\eqref{eq:system_def} can affect the fast and/or the slow component. In what follows, we construct an approximation that consists \ref{def:_drift} in considering that the slow component is not stochastic but evolves deterministically according to its drift (which is its average change), and \ref{def:_transition_kernel} in using a time-averaging method that shows the stochastic process $\bY$ as being in some stationary state given $\bx$. This leads us to two definitions:
\begin{enumerate}[label=(\roman*)]
    \item \label{def:_drift} We call the drift of the slow system (or more concisely the drift) the average change of $\bX_s$. It is the sum over all possible transitions of the rate of transition multiplied by the changes that such transitions induce on $\bX_t$.  By the form of the transitions in \eqref{eq:system_def}, if the process starts in $(\bX_t,\bY_t)=(\bx,\by)$, the drift is given by:
    \begin{align}
        F(\bx,\by) := 
        \sum_{\ell,\by'} \alpha_{\ell,\by'}(\bx,\by) \ell  \in\R^{d_x}. \label{eq:drift}
    \end{align}
    This drift function depends on the state of the fast system $\by$.
    \item \label{def:_transition_kernel} For the fast component, we define a transition kernel $K_{\by,\by'}(\bx)$ that is the rate at which the process $\bY$ jumps from $\by$ to $\by'\neq\by$ (divided by $N$), with the usual convention that $ K_{\by,\by}(\bx) = - \sum_{\by'\ne \by} K_{\by,\by'}(\bx)$:
    \begin{align}
        \label{eq:K}
        K_{\by,\by'}(\bx) = \sum_{\ell} \alpha_{\ell, \by'}(\bx, \by).
    \end{align}
    As $\calY$ is finite, for a fixed $\bx$, $K(\bx)$ is a matrix that corresponds to the kernel of a continuous time Markov chain. Our assumptions will imply that for all $\bx$, the process associated with $K(\bx)$ has a unique stationary distribution, which we denote by the vector $\pi(\bx)=(\pi_{\by}(\bx))_{\by\in\calY}$.
\end{enumerate}

Based on the drift \eqref{eq:drift} of the stochastic system, we define its `average' version $\barF$ by averaging over the stationary distribution of the fast component. That is:
\begin{align*}
    \barF(\bx) := \sum_y \pi_y(\bx) F(\bx,\by).
\end{align*}
For an initial state $\bx$ and $t\geq0$, we call the mean field approximation the solution $\phi_t(\bx)$ of the initial value problem
\begin{align}
    \frac{d}{dt}\phi_t(\bx) = \bar{F}(\phi_t(\bx)), \qquad \phi_0(\bx) = \bx. \label{eq:ODE}
\end{align}
Such an approximation is also called a fluid approximation.

\subsection{Main assumptions}
\label{sec:assumptions}

As we show later, under mild regularity conditions on the transition rate functions $\alpha$, the mean field approximation captures the dynamics of $\bX_t$ well and with a decreasing bias of order $1/N$: 
\begin{align*}
    \esp{h(\bX_t)\mid \bX_0,\bY_0 = \bx,\by} - h(\phi_t(\bx))= \frac{C_h(t)}{N} + o(1/N),
\end{align*}
for a sufficiently regular $h$. This holds for any finite $t$ under assumption \ref{assum:continuous}-\ref{assum:irreducible} below. It also holds for the steady-state regime $t=+\infty$ under the additional assumption \ref{assum:exp_stable}. For the steady-state we also show that $C_h$ can be computed numerically and use it to propose a refined approximation. To obtain these results for finite time, we will assume that:
\begin{enumerate}[label=($A_\text{\arabic*}$)]
    \item \label{assum:continuous} The set of transitions $\calT$ is finite and for all $\ell,\by'\in\calT$, $\alpha_{\ell,\by'}\in\calD^2(\calX\times\calY)$, where $\calD^k(\calX\times\calY)$ the space of functions from $\calX\times\calY$ to $\R$ for which the Hölder norm  $\|h\|_{k,1}$ is finite\footnote{i.e., functions that are $k$-times differentiable with Lipschitz continuous derivatives (see a more precise definition of this norm in Definition~\ref{def:holder_norm} of Appendix~\ref{appdx:semi_group_commute})}.
    \item \label{assum:irreducible} For all $\bx\in\calX$, the matrix $K(\bx)$ defined in Equation~\eqref{eq:K} has a unique irreducible class.
    \item \label{assum:exp_stable} The ODE \eqref{eq:ODE} has a unique, exponentially stable equilibrium, that we denote by $\Phinf$, i.e., there exist $a,b > 0$ s.t. $\norm{ \phi_t(\bx) - \Phinf } \leq a \exp(-bt)$ for all $\bx\in \calX$.
\end{enumerate}
Assumptions \ref{assum:continuous} and \ref{assum:exp_stable} are classical to ensure mean field convergence results. As stated in \cite{gast_refined_2017,ying_approximation_2016}, requiring that the transition rates are twice differentiable which is necessary to guarantee the existence of derivatives for the drift and the differential equation needed for the proofs of the theorems.
Assumption \ref{assum:exp_stable} ensures the existence of a unique equilibrium point to which all trajectories of the differential equation converge. This classical assumption is essentially needed to show that the stationary distribution of the stochastic process converges to a deterministic limit, see \cite{benaim_class_2008}. It guarantees the stability of both the ODE and the 'slow' Poisson equation used in the proof.

By Assumption~\ref{assum:irreducible}, we mean that for all $\bx$ the Markov chain should have a unique subset of states that is irreducible (there can be additional states, but they should all be transient). This assumption is equivalent to assuming the uniqueness of the stationary distribution of the Markov chain induced by the generator matrix $K(\bx)$ which is essential to define the `averaged' drift and mean field approximation. For a given $\bx$, the stationary distribution $\pi(\bx)$ will be non-zero for all states that are in the irreducible component ($\pi_{\by}(\bx)>0$ for such $\by$'s) and will be zero for the others ($\pi_{\by}(\bx)=0$ for all states that are transient for $K(\bx)$).  This assumption is slightly more general than assuming that $K(\bx)$ is irreducible because it allows for transient states.

We will show later that the assumptions \ref{assum:continuous} and \ref{assum:irreducible} imply that $\bar{F}$ is Lipschitz continuous because they imply that $\pi(\bx)$ is Lipschitz-continuous (see Lemma~\ref{lemma:g-inverse_poisson-solution}).  This implies that the mean field approximation \eqref{eq:ODE} is well defined.

\section{Main Results}
\label{sec:main_results}

This section includes our main results which are threefold. In \ref{sec:transient} we obtain accuracy results for the mean field approximation in the transient regime. In \ref{sec:steady-state} we obtain comparable results when the stochastic system is in its steady-state. Lastly, in \ref{sec:refinement} we introduce a correction term for the steady-state, give accuracy bounds and display closed form expressions of the corrections.

\subsection{Transient Regime}
\label{sec:transient}

Our model is a two timescale model, which makes it amenable to be analyzed by time-averaging methods such as the one used in \cite{benaim_class_2008,robert_stochastic_2021,cecchi_mean-field_2019,castiel_induced_2021}. Such methods guarantee that the stochastic process $\bX_t$ converges to the solution of the ODE given by \eqref{eq:ODE}. Yet, most of the papers that use averaging methods do not quantify the rate at which this convergence occurs. 
Our first result, Theorem~\ref{thrm:stoch_sys_ode}, states that the difference between the `average' mean field approximation $\phi$ derived from the average drift $\barF$ approximates the average behavior of the slow component of the two-timescale system with a bias asymptotically equivalent to $C_h(t)/N$.  This result is similar to the one obtained in \cite{gast_size_2019} for classical density dependent process.


\begin{theorem} \label{thrm:stoch_sys_ode}
    Consider the two-timescale stochastic system $(\bX\toN,\bY\toN)_{t\geq0}$ as introduced in Section~\ref{sec:model} starting at $\bX_0\toN,\bY_0\toN = \bx, \by$, and recall that $\phi(\bx)$ is the solution of the ODE~\eqref{eq:ODE} with initial condition $\bx$. Further, assume \ref{assum:continuous} and \ref{assum:irreducible}. For any $h\in\calD^2(\calX)$ and $t>0$, there exists a constant $C_h(t)$ such that: 
    \begin{align*}
        \lim_{N\rightarrow\infty} N\left(\esp{h(\bX\toN_t)\mid \bX_0\toN,\bY_0\toN = \bx, \by} - h\bigl(\phi_t(\bx)\bigr)\right) = C_h(t).
    \end{align*}
\end{theorem}
\begin{proof}[Main element of the proofs]
    The full proof of this theorem is given in Section~\ref{proof:transient}. It is decomposed into two parts that correspond to the two approximations (scaling and averaging):
    \begin{itemize}
        \item The first is to approximate the infinitesimal generator $Lh(\bX_t,\bY_t)$ (defined in Section~\ref{sec:notations}) of the stochastic system by a process whose drift is $D_x h(\bX_t) \cdot F(\bX_t,\bY_t)$. This part uses that the rate functions $\alpha$ and $h$ are differentiable with respect to the slow variable $\bx$. 
        \item The second is that the actual drift of the mean field approximation is $\bar{F}(\bx)=\sum_{\by}\pi_{\by}(\bx)f(\bx,\by)$ and not the $F(\bx,\by)$ obtained from the stochastic system. This leads us to bound a term of the form $\bar{F}(\bx)-F(\bx,\by)$.          
    \end{itemize}
    The first error term is bounded by using generator techniques similar to the ones used for classical mean field models as in \cite{gast_expected_2017}. It is treated in Section \ref{ssec:proof_transient_drift}. To bound the second term, we use averaging ideas similar to the ones of \cite{benaim_class_2008} that are related to how fast the fast timescale converges to its stationary distribution. This is dealt with in Section \ref{ssec:proof_transient_average_drift}. 
\end{proof}

\subsection{Steady-State Results}
\label{sec:steady-state}

Theorem~\ref{thrm:stoch_sys_ode} guarantees that the mean field is a good approximation for any finite time interval. In order to obtain a similar result for the stationary case, an almost necessary condition is that the ODE~\eqref{eq:ODE} has a unique fixed point to which all trajectories converge \cite{benaim_class_2008}. To obtain the equivalent of Theorem~\ref{thrm:stoch_sys_ode}, we assume that this unique fixed point is exponentially stable, as is classically done to obtain steady-state guarantees \cite{gast_refined_2017,ying_approximation_2016}. This assumption is summarized in \ref{assum:exp_stable} and leads to the following result.

\begin{theorem} \label{thrm:steady-state}
    Assume \ref{assum:continuous} - \ref{assum:exp_stable}, and assume that for all $N$, the stochastic system has a stationary distribution. Denote by $(\Xinf\toN, \Yinf\toN)$ a pair of variables having this stationary distribution. Then, for $h\in \calD^2(\calX\times\calY)$ there exists a constant $C_h$ such that: 
    \begin{align*}
        \lim_{N\to\infty} N \left( \esp{h(\Xinf\toN, \Yinf\toN)} - \sum_{\by}\pi_{\by}(\Phinf)h(\Phinf,\by) \right) = C_h.
    \end{align*}
\end{theorem}

\begin{proof}[Main elements of the proof] 
    A detailed proof is provided in Section~\ref{proof:steady-state}.
    To prove the result, the first essential step is to generate two Poisson equations:
    \begin{enumerate}
    \item The first is used to study $\E[ h(\Xinf, \Yinf) - \sum_{\by}\pi_{\by}(\Xinf)h(\Xinf,\by)]$, which is the difference between the true expectation and a hypothetical case where $\Yinf$ would be independent from $\Xinf$  and distributed according to the stationary distribution $\pi(\Xinf)$.
    \item The second is used to study $\E[ \sum_{\by}\pi_{\by}(\Xinf)h(\Xinf,\by)- \sum_{\by}\pi_{\by}(\Phinf)h(\Phinf,\by)]$, which is the error of the 'slow' process compared to the `average' mean field.
    \end{enumerate}
    The rest of the proof treats these Poisson equations to establish the constant $C_h$. This is done by generator comparison techniques and application of derivative bounds for the solutions of the Poisson equation.
\end{proof}
We note that contrary to Theorem~\ref{thrm:stoch_sys_ode}, Theorem~\ref{thrm:steady-state} allows for functions $h$ that can depend on both the slow and the fast components: $\bx$ and $\by$. In fact, Theorem~\ref{thrm:stoch_sys_ode} would not be true if we allow functions $h$ to depend on $\bY_t$ as in Theorem~\ref{thrm:steady-state}, because of the fast transitions of $\bY_t$. For the steady-state, we use that $(\Xinf,\Yinf)$ starts in steady-state.

In Theorem~\ref{thrm:refinement} below, we will show that, for the steady-state and for functions that only depend on $\bx$, we can go further and propose an almost closed-form expression for the term $C_h$. This new bias term provides a refined accuracy of order $O(1/N^2)$. To prove this result, we will use that Theorem~\ref{thrm:steady-state} allows functions $h$ which depend on $\bX$ and $\bY$ to show that this new approximation is $O(1/N^2)$-accurate. We show below that in fact the constant $C_h$ can be computed by a numerical algorithm, and can therefore be used to define a \emph{refined approximation}, similarly to what is done for classical mean field model in \cite{gast_refined_2017}.

\subsection{Steady-State Refinement}
\label{sec:refinement}

To obtain a refined approximation, we utilize ideas introduced in \cite{gast_refined_2017} and propose an almost-closed form expression for the term $C_h$ of Theorem~\ref{thrm:steady-state}. As we will see later in the proofs, the bias correction term is composed of two distinct components:
\begin{enumerate}
    \item The first one (terms $V$ and $W$) is the analogue of the $V$ and $W$ terms of \cite{gast_refined_2017} and corresponds to the difference between the stochastic jumps of the slow system versus having a ODE corresponding to the (non-averaged) drift $f(\bx,\by)$.
    \item The second component (terms $T$, $S$ and $U$) corresponds to approximating $\bY$ by its stationary distribution $\pi_{\by}(\bx)$, and its consequence on the behavior of the slow system $\bX$. 
\end{enumerate}
The proofs of the expression for $V$ and $W$ are essentially those derived in \cite{gast_refined_2017}. The second correction component (terms $T$, $S$ and $U$) is related to the difference between drift and its average version. It involves studying the intricate and coupled dynamics of $\bX$ and $\bY$ which, to the best of our knowledge, has not been studied and yields novel results.  

\begin{theorem}
    \label{thrm:refinement}
    Assume \ref{assum:continuous} - \ref{assum:exp_stable}. Then, there exist vectors $V$, $S$, $T$ and matrices $W$ and $U$ that are solutions of linear systems of equations such that for $h\in\calD^2(\calX)$, we define $C_h$ of Theorem~\ref{thrm:steady-state} as:
    \begin{align*}
        C_h = \sum_i \frac{\partial h}{\partial x_i}(\Phinf) \left(V_i + T_i + S_i\right) + \frac12 \sum_{i,j} \frac{\partial^2 h}{\partial x_i x_j}(\Phinf) \big(W_{i,j} + U_{i,j} \big).
    \end{align*} 
\end{theorem}

\begin{proof}
    This theorem is a consequence of Proposition~\ref{prop:refinement} that shows the existence of the above constants, combined with Proposition~\ref{prop:closed_form_G_slow} which provides the linear equations satisfied by $V$ and $W$, and Proposition~\ref{prop:closed_form_J} which provides the equations satisfied by $T$, $U$ and $V$.
\end{proof}

We would like to emphasize that the result of Theorem~\ref{thrm:refinement} is only valid for functions $h$ that do \emph{not} depend on $\by$.  This shows that there exists a computable constant $C_h$ such that, in steady-state:
\begin{align*}
    \esp{h(\bX_\infty)} = \Phinf + \frac{C_h}{N} + o(1/N). 
\end{align*}
Similar to \cite{gast_refined_2017}, we call $\Phinf + \frac{C_h}{N}$ the \emph{refined approximation} of $\esp{h(\Xinf)}$. As we see in our numerical experiment in Section~\ref{sec:example_csma_model}, this constant is computable and can provide a more accurate approximation than the classical mean field. 

In fact, combined with Theorem~\ref{thrm:steady-state}, we can show that the $o(1/N)$ term is a $O(1/N^2)$ term. This shows that the refined approximation is $O(1/N^2)$-accurate. Note, for functions $h$ that do depend on $\by$, the existence of such a function is guaranteed by Theorem~\ref{thrm:steady-state} but a closed-form expression is currently out of our reach. 


\section{Proofs}
\label{sec:proofs}


\subsection{Stochastic Semi-Groups and Generators}
\label{sec:notations}


Given the stochastic process of Section~\ref{sec:def_stoch_sys}, we define the stochastic semi-group operator which maps a pair of initial values $(\bx,\by)$ and a function $h$ to the expected value of the system at time $t$. This semi-group $\psi_t$ associates to a function $h:\calX\toN \times \calY\to\R$ the function $\psi_th : \calX\toN \times \calY\to\R$ defined as: 
\begin{align}
    \label{eq:def_phi_t}
    \psi_t h(\bx,\by) = \E[h(\bX_t, \bY_t) \mid X_0, Y_0 = \bx,y ].
\end{align}
Using the right continuity of the slow-fast system, it is easy to verify that $\psi_t$ is indeed a $C_0$-semi-group (see a more precise definition in Definition \ref{def:semi-group} in Appendix~\ref{appdx:semi_group_commute}).


The infinitesimal generator of the stochastic process is the operator $L$ that maps a function $h\in\calD^k(\calX\times\calY)$ to $Lh:\calX \times\calY \to \R$ defined by:
\begin{align}
    Lh(\bx,\by) & = \sum_{\ell, \by'\in \calT} N \alpha_{\ell,\by'}(\bx,\by) \left( h(x + \ell / N , \by') - h(\bx,\by) \right). \label{eq:inf_generator_eq}
\end{align}
Note that $Lh$ is obtained by considering average infinitesimal change of the stochastic system starting in $(\bx,\by)$, i.e., 
\begin{align*}
    Lh (\bx,\by)=\lim_{t \downarrow 0} \bigl( \psi_{t} h(\bx,\by) - \psi_0 h(\bx,\by)\bigr)/t = \lim_{t\downarrow0}\bigl(\E[ h(\bX_{t}, \bY_{t}) \mid X_0,Y_0 = \bx,\by ] - h(\bx,\by)\bigr)/t.
\end{align*} 


Similarly to the notations of the semi-group and generator of the stochastic process, for a given function $h\in\calD^k(\calX)$, we denote by $\Phi_t h (\bx) := h(\phi_t(\bx))$ the $C_0$-semi-group corresponding to ODE~\eqref{eq:ODE}. For differentiable $h$, the infinitesimal generator is given by 
\begin{align}
\Lambda h(\bx) := D_xh(\bx) \barF(\bx). \label{eq:ode_generator}
\end{align}
To strengthen intuition, the time derivative of $h(\phi_t(\bx))$ can be expressed as 
\begin{align}
\frac{d}{dt} h(\phi_t(\bx)) = D_x h(\phi_t(\bx))\barF(\phi_t(\bx)) = \Phi_t\Lambda h(\bx). \label{eq:ode_time_deriv}
\end{align} 
By commuting $\Lambda$ and $\Phi_t$ we have $\Lambda \Phi_t h(\bx) = D_x (h \circ \phi_t)(\bx) \barF(\bx)$ which is equal to \eqref{eq:ode_time_deriv} by the results of Appendix~\ref{appdx:semi_group_commute}. We will use this property to prove the theorems of the following section.

\paragraph*{Abuse of Notations and Dependence on the Fast Component} 

The semi-group and generator of the stochastic system are generally defined for functions $h$ in $\calD^k(\calX\times\calY)$, i.e., functions which depend on the slow and fast component whereas the semi-group and generator of the ODE are defined for $h$ in $\calD^k(\calX)$ that do \textit{not} depend on the fast component. 
What we refer to as abuse of notation is the notation we use for the mapping $\psi_t$ of a function $h\in\calD^k(\calX)$. $\psi_t h$ still depends on the fast component even if $h$ does not (since the evolution $\bX_t$ depends on the state of $\bY_t$ and the initial values $(\bx,\by)$):
\begin{align}
    \psi_t h(\bx,\by) = \E[h(\bX_t) \mid \bX_0, \bY_0 = \bx,\by ], \label{eq:stoch_semi_gr_abuse_of_notation} 
\end{align}
and 
\begin{align*}
    L h(\bx,\by) & = \sum_{\ell, \by'} N \alpha_{\ell,\by'}(\bx,\by) \bigl( h(\bx + \ell / N) - h(\bx) \bigr).
\end{align*}
In essence, this allows to use the notion of the semi-group and generator to functions of the slow process. 
For consistency, we do the same for the ODE: for an arbitrary $\by\in\calY$, $\phi_t(\bx,\by) := \phi_t(\bx)$ and therefore $\Phi_t h (\bx,\by) = h(\phi_t(\bx,\by)) = h(\phi_t(\bx))$ which is motivated by the abuse of notation in \eqref{eq:stoch_semi_gr_abuse_of_notation}. This will merely be used in the proofs and allows focusing on integral proof ideas instead of complex notations.


\subsection{Generator of the Fast Process and Regularity of the Poisson Equation}
\label{sec:fast_poisson_regularity}

The generator $L$ describes the changes induced by the transitions of the fast and slow process. In our analysis, it will be useful to analyze the changes due to the jumps in $Y$ only.  We denote by $\Lf$ the generator of the Markov chain induced by $K(\bx)$ that takes as input a function $h\in\calD^k(\calX\times \calY)$ and associates another function $\Lf h$ defined by:
\begin{align*}
    \Lf h(\bx,\by) &= \sum_{\by'}K_{\by,\by'}(\bx)(h(\bx,\by')-h(\bx,\by)).
\end{align*}
Compared to \eqref{eq:system_def}, $\Lf$ is independent of $N$ because the transition rates are rescaled by $1/N$. 

Under Assumption~\ref{assum:irreducible}, for all $\bx$, the matrix $K(\bx)$ characterizes a Markov chain on the state space $\calY$ that has a unique stationary distribution denoted by $\pi(\bx)$. $\pi_y(\bx)$ is the stationary probability of $y\in\calY$ of the Markov chain induced by the kernel $K(\bx)$. 
Subsequently, we will study the distance between the 'true' stochastic process $\bY_t$ and an averaged system where the distribution of $\bY_t$ is replaced by the stationary distribution $\pi(\bx)$. To quantify the error made when replacing $\bY_t$ with the stationary distribution, we consider the following Poisson equation:
\begin{align}
    \label{eq:PoissonY}
    h(\bx,\by) - \sum_{y\in\calY} h(\bx,\by)\pi_{y}(\bx) = \Lf \Gfast_h(\bx,\by).
\end{align}
For a given function $h:\calX\times\calY\to\R^n$ ($n$ is arbitrary but finite), a function $\Gfast_h:\calX\times\calY\to\R^n$ that satisfies the above equation is called a solution to this Poisson equation. We will have particular interest in $h(\bx,\by)=F(\bx,\by) \in \R^{d_x}$ namely when the drift $F(\bx,\by)$ is compared to its `average' version $\barF(\bx)=\sum_y F(\bx,\by) \pi(\by)$.

The existence of a regular solution to this Poisson equation is guaranteed by the following Lemma~\ref{lemma:g-inverse_poisson-solution}. Note that the solution of the above Poisson equation is not unique: If $\Gfast_h$ is a solution, then for any constant $c\in\R^n$, a function $\Gfast_h + c$ is also a solution. Later in the proofs of the theorems, when we talk about `a solution of the Poisson equation', we refer to the solution given in Lemma~\ref{lemma:g-inverse_poisson-solution}. 

\begin{lemma}\label{lemma:g-inverse_poisson-solution}
    Assume \ref{assum:irreducible}.  Then for all $\bx\in\calX$:
    \begin{enumerate}
        \item The Markov chain corresponding to $K(\bx)$ has a unique stationary distribution that we denote by $\pi(\bx)$. We denote by $\Pi(\bx)=\mathds{1}\ \pi^T(\bx)$ the matrix where each line is equal to $\pi(\bx)$.
        \item  The matrix $(K(\bx)+\Pi(\bx))$ is invertible and its inverse is a generalized inverse of $K(\bx)$. 
        \item Define $K^+(\bx) = (K(\bx)+\Pi(\bx))^{-1} (I - \Pi(\bx))$, then, for all functions $h: \calX \times \calY \to \R^n$, $G_h(\bx,\by)=\sum_{\by'}K^+_{\by,\by'}(\bx)h(\bx, \by')$ is a solution of the Poisson Equation~\eqref{eq:PoissonY}.
    \end{enumerate}
    If, in addition, Assumption~\ref{assum:continuous} holds, then $K^+(\bx)$ is twice differentiable in $\bx$.
\end{lemma}

The proof is provided in Appendix~\ref{apx:proof_poisson}. Note in particular, this result implies that if $h$ is (twice) differentiable in $\bx$ then the same holds true for $\Gfast_h$.  



\subsection{Proof of Theorem~\ref{thrm:stoch_sys_ode} - Transient State Proof}
\label{proof:transient}

The proof of Theorem~\ref{thrm:stoch_sys_ode} can be decomposed in two main parts. We first use a generator transformation to show that the slow system is well approximated by a system whose drift is $D_x h(\bX_t) \ F(\bX_t,\bY_t)$ (\ref{ssec:proof_transient_drift}). This leads us to treat terms of the form $\bar{F}(\bX_t) - F(\bX_t,\bY_t)$. To study them, we use the solution of the Poisson equation for the fast system \eqref{eq:PoissonY}. 
The second part is the more technical and novel. It is detailed in \ref{ssec:proof_transient_average_drift}. Some technical lemmas are postponed to Appendix~\ref{apx:lemma_transient}.

\subsubsection{Error due to replacing the stochastic jumps of $\bX_t$ by the drift}
\label{ssec:proof_transient_drift}

For $h\in\calD^2(\calX)$ recall that the $C_0$-semi-groups of the stochastic system and of the ODE are defined as
\begin{align*}
    \psi_sh(\bx_0,\by_0) = \E[ h(\bX_s) \mid  \bX_0,\bY_0 = \bx_0,\by_0] \quad \text{ and } \quad \Phi_t h (\bx_0,\by_0) = h(\phi_t(\bx_0)).
\end{align*}
We define $\nu_s h(\bx,\by) := \psi_s \Phi_{t-s}h(\bx,\by)$ and rewrite 
\begin{align}
\E[ h(\bX_t) - h(\phi_t(\bx)) \mid \bX_0,\bY_0 = \bx_0,\by_0] = \nu_t h(\bx_0,\by_0) - \nu_0 h(\bx_0,\by_0)= \int_0^t \frac{d}{ds} \nu_s h(\bx_0,\by_0)ds. \label{eq:int_form_of_difference}
\end{align}
To show that the last equation indeed holds true, observe that 
\begin{align}
\frac{d}{ds}\nu_s h(\bx_0,\by_0) &= \frac{d}{ds}\psi_s \Phi_{t-s}h(\bx_0,\by_0) \nonumber\\
    &= L \psi_s \Phi_{t-s} h(\bx_0,\by_0) - \psi_s \Lambda  \Phi_{t-s} h(\bx_0,\by_0). \label{eq:generator_comparison}
\end{align}
By regularity assumptions on the transition rates and bounded state space the above equation is finite for all $(\bx,\by)\in \calX\times \calY$ and time $s\geq0$. The dominated convergence theorem thus justifies the interchange of derivative and integral and validates the last equality. 
As pointed out in Appendix~\ref{appdx:semi_group_commute}, $\psi_s$ and $L$, the stochastic semi-group and its infinitesimal generator, commute, i.e., $L \psi_s = \psi_s  L$. Hence,
\begin{align}
    \eqref{eq:generator_comparison} & = \psi_s (L - \Lambda) \Phi_{t-s} h(\bx_0,\by_0) \nonumber\\
    & = \E[ (L - \Lambda) \Phi_{t-s}h (\bX_s,\bY_s) \mid \bX_0, \bY_0 = \bx_0,\by_0],\label{eq:L-Lambda}
\end{align}
where the last line follows by definition of $\psi_s$.
Let $g(\bx)=h(\phi_{t-s}(\bx))$. $g$ is twice differentiable with respect to the initial condition $\bx$ as it lies in $\calD^2(\R^d,\R)$ by Lemma~\ref{lem:ode_smooth}. By the use of Taylor expansion rewrite $Lg$ as
\begin{align*}
    L g (\bx,\by) & =  \sum_{\ell,\by'\in\calT} N\alpha_{\ell, y'}(\bx,\by)(g(\bx+\frac{\ell}{N}) - g(\bx))\\
    &= \sum_{\ell,\by'\in\calT} \alpha_{\ell, \by'}(\bx,\by) D_x g(\bx)\ell
    +\frac1N\sum_{\ell,\by'\in\calT} \alpha_{\ell, \by'}(\bx,\by) D^2_x g(\bx)\cdot(\ell,\ell) + o(\frac{\norm{g}_{2,1}\sup_\ell \abs{\ell}^2}{N})\\
    & = Dg(\bx) F(\bx,\by)+\frac1N\sum_{\ell,\by'\in\calT} \alpha_{\ell, \by'}(\bx,\by) D^2_x g(\bx)\cdot(\ell,\ell) + o(\frac{C_R\norm{g}_{2,1}\sup_\ell \abs{\ell}^2}{N}),
\end{align*}
with $D^2g(\bx)\cdot(\ell,\ell) = \sum_{m,n} \frac{\partial^2 g}{\partial x_m \partial x_n}(\bx) \ell_m \ell_n$. The convergence depends further on the Hölder norm of $h(\phi(x))$, bounds on the transition rates $C_R$ and jump sizes $\sup \abs{\ell}$.
The generator of the ODE related semi-group, $\Lambda$, introduced in Equation~\eqref{eq:ode_generator}, maps $g$ to $\Lambda g(\bx,\by) = D_x g(\bx) \barF(\bx)$. Hence, we have 
\begin{align*}
    (L-\Lambda)g(\bx,\by) &= D_x g(\bx)(F(\bx,\by) - \bar{F}(\bx)) + \frac1N\sum_{\ell,y'\in\calT} \alpha_{\ell, y'}(\bx,\by) D^2_x g(\bx)\cdot(\ell,\ell) + o(1/N).
\end{align*}

\subsubsection{Error due to replacing the drift by the average drift}
\label{ssec:proof_transient_average_drift}

Next, we have a closer look at the first summand of the right hand side in the above equation. Denote by $\Gfast_F$ a solution of the Poisson equation~\eqref{eq:PoissonY} where the function ``$h$' of \eqref{eq:PoissonY} is set to the drift $F$. Since $D_x g(\bx)$ does not depend on $\by$ we have by definition of the Poisson equation and $\Lf$:
\begin{align}
    D_xg(\bx)(F(\bx,\by) - \bar{F}(\bx)) = D_x g(\bx) \Lf \Gfast_F(\bx,\by) = \Lf D_x g(\bx) \Gfast_F(\bx,\by).
    \label{eq:transient_poisson}
\end{align}
Combining the above with equation~\eqref{eq:L-Lambda} and plugging everything into the integral of equation~\eqref{eq:int_form_of_difference}, we get:
\begin{align}
    \label{eq:transient_1}
    \E[ h(\bX_t) - h(\phi_t(\bx_0)) ] = & \int_0^t\E[ \Lf D_x(h\circ\phi_{t-s})(\bX_s)\cdot \Gfast_F(\bX_s,\bY_s)] \nonumber \\  
    & + \frac1N \E[ \sum_{\ell,\by'\in\calT} \alpha_{\ell, \by'}(\bX_s,\bY_s) D^2_x (h\circ\phi_{t-s})(\bX_s)\cdot(\ell,\ell)]ds + o(1/N),
\end{align}
\raggedright
where we suppress the conditioning on the initial values $\bX_0,\bY_0=\bx_0,\by_0$. 
Let $H_s(\bx,\by) := \esp{\Hshort{s}\mid \bX_0,\bY_0=\bx,\by}$. Applying Lemma~\ref{lemma:h_s} with $g_s(\bX_s,\bY_s)=\Hshort{s}$ implies that
\begin{align}
    0 = \frac1N H_t(\bx,\by) & - \frac1N H_0(\bx,\by) - \int_0^t \E[ \frac1N D_x\bigl( D_x h(\phi_{t-s}(\bX_s)) \cdot \barF (\phi_{t-s}(\bX_s)) \bigr) \cdot \Gfast_F(\bX_s,\bY_s) ] ds \notag \\
& \phantom{=\int} + \E [ \frac1N L \Hshort{s}] ds. \label{eq:artificial_zero_trans}
\end{align}
As $h$ and $\phi_{t-s}$ are twice continuously differentiable, by compactness of $\calX$, and as $\Gfast_F$ is finite, $D_x \bigl( D_x h(\phi_{t-s}(\bX_s) ) \barF (\phi_{t-s}(\bX_s))\bigr) \Gfast_F(\bX_s,\bY_s)$ is bounded. Moreover, by Lemma \ref{lemma:lfast_bound} we have 

\begin{align*}
& (\Lf - \frac1N L)\Hshort{s} \\
& = \frac1N \sum_{\ell, \by'} \alpha_{\ell, \by'}(\bX_s, \bY_s) \ell \ D_x ( D_x( h \circ \phi_{t-s}) (\bX_s) \Gfast_F(\bX_s, \by')) + o(1/N). \\
\end{align*} 

Define 
\begin{align*}
& C_h(t) = \int_0^t \E [ \sum_{\ell,\by'\in\calT} \alpha_{\ell, \by'}(\bX_s,\bY_s) D^2_x h(\phi_{t-s}(\bX_s))\cdot(\ell,\ell)] ds \\
& + \int_0^t \E [ D_x\bigl( D_x h(\phi_{t-s}(\bX_s)) \cdot \barF (\phi_{t-s}(\bX_s)) \bigr) \cdot \Gfast_F(\bX_s,\bY_s) ]ds \\
& + \int_0^t \E[\sum_{\ell, \by'} \alpha_{\ell, \by'}(\bX_s, \bY_s) \ell \ D_x ( D_x( h \circ \phi_{t-s}) (\bX_s) \Gfast_F(\bX_s, \by')) ] ds  
+ \esp{H_t(\bX_t,\bY_t)} - H_0(\bx,\by),
\end{align*}
which lets us rewrite
\begin{align*}
N \left( \E[ h(\bX_t)] - h(\phi_t(\bx)) \right) = C_h(t) + o(1).
\end{align*}
Taking the limit $N\to\infty$ concludes the proof.


\subsection{Proof of Theorem~\ref{thrm:steady-state} - Steady-State Proof}
\label{proof:steady-state}

\raggedright
Considering $\E[ h(\Xinf, \Yinf) - \sum_{\by}\pi_y(\Phinf)h(\Phinf,\by)]$, first, add and subtract the term $\E[\sum_{\by}\pi_y(\Xinf)h(\Xinf,\by)]$ into the equation, yielding
\begin{align*}
 \E[ h(\Xinf, \Yinf) - \sum_{\by}\pi_{\by}(\Phinf)h(\Phinf,\by)] 
&= \E[ h(\Xinf, \Yinf) - \sum_{\by}\pi_{\by}(\Xinf)h(\Xinf,\by)] \tag{A} \label{eq:A} \\
& \phantom{=}+ \E[ \sum_{\by}\pi_{\by}(\Xinf)h(\Xinf,\by)- \sum_{\by}\pi_{\by}(\Phinf)h(\Phinf,\by)]. \tag{B} \label{eq:B}
\end{align*}
Treating the two terms \eqref{eq:A} and \eqref{eq:B} separately, we define $\Gfast_h$ and $\Gslow_h$ as the solutions to the Poisson equations
\begin{align*}
\Lf \Gfast_h(\bx,\by) & = h(\bx, \by) - \sum_{\by}\pi_{\by}(\bx)h(\bx,\by), \tag{`fast' Poisson Equation} \label{eq:Poisson:A} \\
\Lambda \Gslow_h(\bx) & = \sum_{\by}\pi_{\by}(\bx)h(\bx,\by) - \sum_{\by}\pi_{\by}(\Phinf)h(\Phinf,\by) . \tag{`slow' Poisson Equation} \label{eq:Poisson:B}
\end{align*}
Recall that the existence and regularity properties of the function $\Gfast_h$ are investigated in Lemma~\ref{lemma:g-inverse_poisson-solution}. For $\Gslow_h$, it is known (e.g., \cite{gast_expected_2017}) that the exponential stability of the unique fixed point \ref{assum:exp_stable} along with the smoothness of $F$ guarantees that $\Gslow_h\in D^2(\calX)$.

Next we use that taking the expectation of the generator $L$ applied to an arbitrary function $g$ over the stationary distribution of  $(\Xinf, \Yinf)$ is zero, i.e., 
$\E[ Lg(\Xinf, \Yinf) ] = 0.$
By using Lemma \ref{lemma:lfast_bound}, it follows
\begin{align*}
\eqref{eq:A} & = \E[ \Lf \Gfast_h(\Xinf, \Yinf) - \frac{1}{N} L \Gfast_h(\Xinf, \Yinf) ] \\
& = \esp{\frac1N \sum_{\ell, \by'} \alpha_{\ell,\by'}(\Xinf, \Yinf) \ell \ D_x \Gfast_h(\Xinf, \by') + o(1/N)}.
\end{align*}
For the second term, we apply the \ref{eq:Poisson:B} equation and subtract $\esp{ L \Gslow_h(\Xinf, \Yinf)} = 0 $ which yields 
\begin{align}
\eqref{eq:B} = \E[\Lambda \Gslow_h(\Xinf)  - L \Gslow_h(\Xinf, \Yinf)]. \label{eq:lemma_second_poisson}
\end{align}
Note that even if $\Gslow_h$ only depends on $\bx$, the function $L\Gslow_h$ does depend on $\bx$ and $\by$ because the rate functions that appear in the generator $L$ do depend on $\bx$ and $\by$.  By using the Taylor expansion, $L\Gslow_h$ equals
\begin{align*}
    L \Gslow_h (\bx, \by) & = \sum_{\ell, \by'} N \alpha_{\ell, \by'}(\bx,\by) \bigl( \Gslow_h(x + \frac{\ell}{N}) -  \Gslow_h(\bx) \bigr) \\
    & =  \sum_{\ell, \by'} N \alpha_{\ell, \by'}(\bx,\by) D_x \Gslow_h(\bx) \frac{\ell}{N}  + \frac12 \sum_{\ell, \by'} N \alpha_{\ell, \by'}(\bx,\by) D_x^2 \Gslow_h(\bx) \frac{(\ell,\ell)}{N^2}  + o(\frac{\norm{\ell}^2}{N}).
\end{align*}
By application of the definition of $F(\bx,\by)$, the first term of the right hand side is equal to $F(\bx,\by) D_x \Gslow_h(\bx)$.
Moreover, by definition, $\Lambda \Gslow_h(\Xinf) = D_x \Gslow_h(\bx) \barF (\bx)$. This shows that 
\begin{align*}
    \E[D_x \Gslow_h(\Xinf) ( \barF (\Xinf) - F(\Xinf, \Yinf) ) + \sum_{\ell, \by'} \alpha_{\ell, \by'}(\Xinf,\Yinf) \bigl(\frac12 D_x^2 \Gslow_h(\Xinf) \cdot \frac{(\ell,\ell)}{N}  \bigr) + o(\frac{\norm{\ell}}{N^2})]. 
\end{align*}
As $\Gslow_h$ has a bounded second derivative, the second term is of order $O(1/N)$. However, to bound the first summand we need to investigate $\barF (\bx) - F(\bx, \by) = \sum_y \pi_{\by'}(\bx) F(\bx,\by') -F(\bx,\by)$. Falling back on \eqref{eq:Poisson:A} allows two write
\begin{align*}
    D_x \Gslow_h(\bx) ( \barF (\bx) - F(\bx, \by) ) = - D_x \Gslow_h(\bx) \Lf \Gfast_F(\bx,\by) = - \Lf  D_x \Gslow_h(\bx) \Gfast_F(\bx,\by).
\end{align*}
The last equality holds due to the definition of $\Lf$, and $D_x \Gslow_h(\bx)$ which depends only on $\bx$. By adding the zero term $\E[\frac{1}{N} L D_x \Gslow_h(\Xinf) \Gfast_F(\Xinf, \Yinf)]$ and using Lemma \ref{lemma:lfast_bound}. We see that 
\begin{align*}
& \E [ - \Lf  D_x \Gslow_h(\Xinf) \Gfast_F(\Xinf, \Yinf) + \frac{1}{N} L D_x \Gslow_h(\Xinf) \Gfast_F(\Xinf,\Yinf) ] \\
& = \E [ \frac1N \sum_{\ell, \by'} \alpha_{\ell, y'}(\Xinf, \Yinf) \ell \ D_x( D_x \Gslow_h(\Xinf) \Gfast_F(\Xinf,\by') ) + o(1/N)].
\end{align*}
Put together, the error term  of \eqref{eq:B} is
\begin{align}
\E [ \frac1N \sum_{\ell, \by'} & \alpha_{\ell, \by'}(\Xinf, \Yinf) \ell \ D_x( D_x \Gslow_h(\Xinf) \Gfast_F(\Xinf,\by') )  + \frac{1}{2N} \sum_{\ell, \by'} \alpha_{\ell, \by'}(\Xinf,\Yinf) D_x^2 \Gslow_h(\Xinf) (\ell,\ell)]. \label{eq:B_residual_term}
\end{align} 
By defining $C_h$ as the sum of the residual term $\esp{\sum_{\ell, \by'} \alpha_{\ell,\by'}(\Xinf, \Yinf) \ell \ D_x \Gfast_h(\Xinf, \by')}$ for \eqref{eq:A} and equation~\eqref{eq:B_residual_term} the by $N$ scaled, the proof concludes.

\subsection{Proof of Theorem~\ref{thrm:refinement} (Refinement Theorem, and Closed form Expressions)}
\label{proof:refinement}

In this section we first decompose the constant $C_h$ in two terms in Proposition~\ref{prop:refinement}. We then study these two terms in Proposition~\ref{prop:closed_form_G_slow} and~\ref{prop:closed_form_J} where we obtain the closed form expressions for the correction terms which allow to numerically obtain the corrections.

We define two functions $Q(x,y)$ and $J_h(x,y)$ (the latter is defined for a function $h\in\calD^2(\calX)$):
\begin{align*}
    Q(\bx,\by) &= \sum_{\ell, \by'} \alpha_{\ell, \by'}(\bx,\by) \ell \ell^T, \\
    J_h(\bx,\by)&= - \sum_{\ell, \by'} \alpha_{\ell,\by'}(\bx,\by) D_x \bigl(D_x\Gslow_h(\bx) \Gfast_F(\bx,\by') \bigr) \ell,
\end{align*}
and we denote by $\bar{J_h}(x) = \sum_y \pi_y(x) J(x,y)$ and $\bar{Q}(x)=\sum_y \pi_y(x)Q(x,y)$ their ``average'' version. The following proposition holds. 
\begin{proposition}
    \label{prop:refinement}
    Assume~\ref{assum:continuous}--\ref{assum:exp_stable}. Then, for any $h\in\calD^2(\calX)$, we have: 
    \begin{align}
        \label{eq:prop_refinement}
        N \bigl( \E[h(\Xinf)] - h(\Phinf) \bigr) =  \frac{1}{2} D_x^2 \Gslow_h(\Phinf) \bar{Q}(\Phinf) + \bar{J_h}(\Phinf) + o(1).
    \end{align}
\end{proposition}


\begin{proof}
Let $h\in\calD^2(\calX)$ and let us denote $\Gslow_h$ the solution of \eqref{eq:Poisson:B}. As $h$ does not depend on $y$, this function is such that for any $x$: $h(x) - h(\Phinf) = \Lambda \Gslow_h(x) = D_x \Gslow_h(x) \bar{F}(x)$. The following steps are similar to the ones for \eqref{eq:B} in the proof of Theorem~\ref{thrm:steady-state}. Hence, applying this for $\Xinf$ and taking the expectation, we get:
\begin{align*}
    N \E[h(\Xinf) - h(\Phinf)] = N \E [ \Lambda \Gslow_h(\Xinf) ] = N \E[ D_x \Gslow_h(\Xinf) \bar{F}(\Xinf) ].
\end{align*}
Recall that for any bounded function $g\in\calD(\calX\times\calY)$, $\esp{Lg(\Xinf,\Yinf)}=0$. Hence,
\begin{align*}
    N \E[ D_x \Gslow_h(\Xinf) \bar{F}(\Xinf) - L \Gslow_h(\Xinf, \Yinf) ].
\end{align*}
Similarly to what we do to prove Theorem~\ref{thrm:steady-state}, we plug the definition of $L$ in the above equation and use a Taylor expansion to show that this equals
\begin{align}
& N \E[ D_x \Gslow_h(\Xinf) \bar{F}(\Xinf) - \sum_{\ell, \by'}N \alpha_{\ell, \by'}(\Xinf,\Yinf) \bigl( \Gslow_h(\Xinf + \frac{\ell}{N}) - \Gslow_h(\Xinf )\bigr)] \nonumber \\
& = N \E[ D_x \Gslow_h(\Xinf) \bar{F}(\Xinf) - \sum_{\ell, \by'}N \alpha_{\ell, \by'}(\Xinf,\Yinf) \bigl( D_x \Gslow_h(\Xinf) \frac{\ell}{N} + \frac12 D_x^2 \Gslow_h(\Xinf) \frac{(\ell,\ell)}{N^2} + o(\frac{\norm{\ell}}{N^2})\bigr)] \nonumber \\
& = N \E[ D_x \Gslow_h(\Xinf) (\bar{F}(\Xinf) - F(\Xinf, \Yinf)) - \frac{1}{2N} D_x^2 \Gslow_h(\Xinf) Q(\Xinf, \Yinf)]  + o(1). \label{eq:proof_ref_sep_error_terms}
\end{align}
The rest of the proof follows using
\begin{itemize}
\item Lemma \ref{lemma:refinement_2} to show $ N \E[ D_x \Gslow_h(\Xinf) ( \bar{F}(\Xinf) - F(\Xinf, \Yinf) )] = \bar{J}(\Phinf) + o(1)$;
\item Theorem \ref{thrm:steady-state} which states  $\E [ \frac{1}{2} D_x^2 \Gslow_h(\Xinf) Q(\Xinf, \Yinf)] = \frac{1}{2} D_x^2 \Gslow_h(\Phinf) \bar{Q}(\Phinf) + o(1)$.
\end{itemize}
Applying the above equations to \eqref{eq:proof_ref_sep_error_terms} implies the statement of Proposition~\ref{prop:refinement}.
\end{proof}

In the rest of the section, we show that the quantity $\frac{1}{2} D_x^2 \Gslow_h(\Phinf) \bar{Q}(\Phinf) + \bar{J}(\Phinf)$ can be easily computed numerically by solving linear systems of equations. As shown in Proposition~\ref{prop:closed_form_G_slow} and~\ref{prop:closed_form_J}, we obtain five different correction terms:
\begin{itemize}
    \item The first two $V\in\R^{d_x}, W\in\R^{d_x \times d_x}$ are closely related to the ones obtained in \cite{gast_refined_2017}[Theorem 3.1] and derived from the term $\frac{1}{2} D_x^2 \Gslow_h(\Phinf) \bar{Q}(\Phinf)$. This term is essentially identical to the refinement term of \cite{gast_refined_2017}.
    \item In Proposition~\ref{prop:closed_form_J} we derive three other refinement terms $S,T \in\R^{d_x}$ and $U\in\R^{d_x \times d_x}$ which give closed form descriptions of $\bar{J}(\Phinf)$. These terms are novel and take into account the difference of the average drift $\barF$ and the actual drift $F$ obtained from the stochastic system. The proof is based on exact expressions for the Poisson equations $\Gfast, \Gslow$ and relies on Lemma~\ref{lemma:g-inverse_poisson-solution}.
\end{itemize}

\subsubsection{Correction terms $V$ and $W$}

To obtain the first result, let $A$ and $B$ be the Jacobian and Hessian matrix of the `average' drift, i.e., 
\begin{align}
    A_{i,j} = \frac{\partial \bar{F}_i}{\partial x_j}(\Phinf) \quad \text{ and } \quad B_{i,k_1,k_2} = \frac{\partial^2 \bar{F}_i}{\partial x_{k_1} x_{k_2}} (\Phinf). \label{eq:refinement_A_B_def}
\end{align}
By the exponential stability of the fixed point (Assumption~\ref{assum:exp_stable}), the matrix $A$ is invertible and the Lyapunov equation $AW+WA^T+\bar{Q}=0$ has a unique solution. We denote by $W$ its solution and define the vector $V$ as
\begin{align*}
    V_i = - \frac12 \sum_j (A^{-1})_{i,j}  \sum_{k_1,k_2} B_{j,k_1,k_2} W_{k_1,k_2}.
\end{align*}
\begin{proposition} \label{prop:closed_form_G_slow}
    Assume \ref{assum:continuous} - \ref{assum:exp_stable} and $h\in \calD^2(\calX)$.  Then: 
    \begin{align*}
        \frac{1}{2} D_x^2 \Gslow_h(\Phinf) \bar{Q}(\Phinf) = \sum_i \frac{\partial h}{\partial x_i} V_i + \frac12 \sum_{i,j} \frac{\partial^2 h}{\partial x_i x_j} W_{i,j} \ .
    \end{align*}
\end{proposition}

\begin{proof}
\label{proof:closed_form_classic}
The proposition is a direct consequence of the results of \cite{gast_refined_2017}[Theorem 3.1] which we apply to the function $\barF(x) = \sum_y \pi_y(x) F(x,y)$. 
\end{proof}


\subsubsection{Correction terms $T$, $S$, $U$}

\begin{proposition}
\label{prop:closed_form_J}
Assume \ref{assum:continuous} - \ref{assum:exp_stable} and $h\in \calD^2(\calX)$.
The closed form solution of $\bar{J}$ at the equilibrium point $\Phinf$ is given by
\begin{align*}
\bar{J}_h(\Phinf) = \sum_i \frac{\partial h}{\partial x_i}(\Phinf) (T_i + S_i) + \sum_{i,j} \frac{\partial^2 h}{\partial x_i \partial x_j}(\Phinf) U_{i,j}.
\end{align*}
$U$ is the unique solution to the Sylvester equation
\begin{align*}
& A X + X A^T  = - O, \qquad \text{ with } \\
& O = \sum_{y} \pi_{y}(\Phinf) \sum_{\ell, y'} \alpha_{\ell,y'}(\Phinf,y) K^+_{y',:}(\Phinf)F(\Phinf,:) \ell^T,
\end{align*}
where $A, B$ are the Jacobian, Hessian of the average drift as defined in \eqref{eq:refinement_A_B_def} and $K^+$ as given in Lemma~\ref{lemma:g-inverse_poisson-solution}.
$T, S$ are defined by
\begin{align*}
& T_i := \sum_{j} A^{-1}_{i,j} \sum_{k_1,k_2} B_{j,k_1,k_2} U_{k_1,k_2}, \\
& S_i := \sum_k A^{-1}_{i,k} \sum_{y} \pi_y(\Phinf) \sum_{\ell, y'} \alpha_{\ell,y'}(\Phinf,y) \left( \sum_{y'} F_k (\Phinf,y') \nabla^T_x K_{y',y'}^{+}(\Phinf) \ell + K_{y',y'}^{+}(\Phinf) \nabla^T_x F_k(\Phinf,y') \ell  \right).
\end{align*}
\end{proposition}

The proof of this proposition is given in Appendix~\ref{apx:proof_closed_form_J}.

\section{Example: CSMA Model}
\label{sec:example_csma_model}

To illustrate our results, we consider the unsaturated CSMA random-access networks studied in \cite{cecchi_mean-field_2019}. In this paper, the authors use a two-scale model to study the performance of a CSMA algorithm with many nodes. The slow process corresponds to the arrival of jobs and the fast process corresponds to the activation and deactivation of nodes. The authors of this paper derive a mean field approximation and show that it is asymptotically exact. With our methods, we go two steps further:
\begin{itemize}
    \item Theorem~\ref{thrm:stoch_sys_ode} and \ref{thrm:steady-state} show that not only the mean field approximation is asymptotically exact but also that the error is only of order $O(1/N)$. 
    \item By using Theorem~\ref{thrm:refinement}, we can compute a refinement term. Our numerical example shows that, similarly to what happens for classical one-scale mean field models \cite{gast_refined_2017}, this refinement term is extremely accurate. It is much more accurate than the classical mean field approximation when the studied system is not too large. 
\end{itemize}

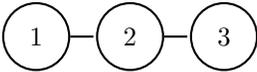
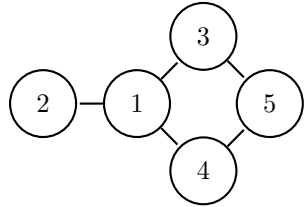
\begin{figure}[ht]
\begin{minipage}{0.45\linewidth}
\begin{subfigure}{\textwidth}
\centering
\begin{tikzpicture}[%
    -, 
    shorten >=1pt,
    node distance=1.25cm,
    thick,
    every state/.style={%
    fill=white,
    draw=black,
    text=black
    }   
    ]

    \node[state] (1) {$1$};

    \node[state] (2) [right of=1] {$2$};
    \node[state] (3) [right of=2] {$3$};
    \path (1)  edge (2)
          (2)  edge (3);
\end{tikzpicture}
\caption{The linear 3 Node Graph.}
\label{fig:linear_graph}
\end{subfigure}
\end{minipage}
\begin{minipage}{0.47\linewidth}
\begin{subfigure}{\textwidth}
\centering
\begin{tikzpicture}[%
    -, 
    shorten >=1pt,
    node distance=1.25cm,
    thick,
    every state/.style={%
    fill=white,
    draw=black,
    text=black
    }
    ]
    \node[state] (3) {$3$};
    \node[state] (1) [below left of=3] {$1$};
    \node[state] (5) [below right of=3] {$5$};
    \node[state] (4) [below right of=1] {$4$};
    \node[state] (2) [left of=1] {$2$};
    \path (1)  edge (2)
          (1)  edge (3)
          (1)  edge (4)
          (3)  edge (5)
          (4)  edge (5);
\end{tikzpicture}
\caption{The 5 Node Graph.}
\label{fig:5_node_graph}
\end{subfigure}
\end{minipage}
\caption{Two examples of interference graphs.}
\label{fig:interference_graphs}
\end{figure}

\subsection{Model Description}

We consider a model with $C$ server types, with $N$ statistically identical servers for each class. All servers communicate through a wireless medium using a random-access protocol and have a finite buffer size $B$. The $C$ classes form an interference graph $G=(\calC, \mathcal{E})$ with $\calC = \{1,\dots,C\}$ the classes and $\mathcal{E}$ the network specific edges. (see for instance Figure~\ref{fig:interference_graphs} for examples of graphs with three or five classes). This interference graph indicates that two servers cannot transmit simultaneously if they either are of the same class or belong to an adjacent class.  For each class $c\in\calC$, we will denote by $y_c\in\{0,1\}$ a variable that equals $1$ if a node of class $c$ is transmitting and $0$ otherwise. We denote by $\calY$ the set of possible activation vectors $y\in\{0,1\}^C$. It is equal to the set of the independent sets of the graph, i.e., all activity vectors for which an active node has only inactive neighbors. For instance, for the graph with three classes which are linearly connected as shown in Figure \ref{fig:linear_graph}, the set of feasible states is given by 
\begin{align*}
    \calY = \{(0,0,0), (1,0,0), (0,1,0), (0,0,1), (1,0,1) \}.    
\end{align*}
Any node can turn active if there are no neighboring active nodes. Once its transmission is finished, it changes back to an idle state. For a given class $c$, we define $\calY_{c}^+ := \{ y \in \calY: y_c = 1\}$ as the subset of states for which node $c$ is active, and by $\calY_{c}^- := \{ y \in \calY : y_{c} = y_{d} = 0 \ \forall d \text{ s.t. } (c,d)\in \mathcal{E} \}$ the subset of states from which node $c$ can turn active, i.e., all neighboring nodes are inactive.  For instance, for the linear $3$ node graph of Figure~\ref{fig:linear_graph} and the class $c=1$ this yields the following sets
\begin{align*}
    \calY_{1}^+ = \{(1,0,0), (1,0,1)\}, && \calY_{1}^- = \{(0,0,0), (0,0,1)\}.
\end{align*}
As in \cite{cecchi_mean-field_2019}, in this non-saturated model, we consider that if a node is in class $c$, new packets arrive to this node at rate $\lambda_c>0$. If a node has a packet to transmit and no neighboring node is transmitting, then this node becomes active at rate $\nu_c>0$. We assume that a transmission from a node of class $c$ takes an exponential time of duration $1/(N\mu_c)>0$, after which the packet leaves the system.

\subsection{Two-scale model representation}
\label{sec:example_rates}
The model as described above fits in our two-timescale representation. To see why, for each class $c\in\calC$ and buffer size $b\in\{0\dots B\}$, we define $X\toN_{t,(c,b)}$ as the fraction of servers of class $c$ that have at least $b$ jobs in their queue at time $t$. We denote by $\bX_t$ the vector of all possible $X_{t,(c,b)}$ for all $c\in\calC$ and $b\in\{0\dots B\}$.  The fast component $\bY_t$ is the activation at time $t$: $Y_{t,(c)}=1$ if a node of class $c$ is transmitting at time $t$ and $0$ otherwise.

Using this representation, we characterize the possible transitions. Given a state pair $(\bx, \by)$, the transitions are represented as a transition vector of the form $(\ell/N, \by')$ and a corresponding transition rate $N \alpha_{\ell,\by'}(\bx,\by)$ such that the state $(\bx,\by)$ jumps to $(\bx +\ell/N, \by')$ at rate $N\alpha_{\ell,\by'}(\bx,\by)$. The transitions can be distinguished into three types:

\begin{itemize}
\item Arrival of a packet to a server of class $c\in \calC$:
\begin{align*}
& \quad {\begin{array}{lcll}  
(\frac{e_{c,i}}{N}, \by) & \text{at rate} & N \lambda_c (1 - x_{c,i}) & \text{for } i = 1,\\
(\frac{e_{c,i}}{N}, \by) & \text{at rate} & N \lambda_c (x_{c,i-1} - x_{c,i}) & \text{for } 2 \leq i \leq B.
\end{array}}
\end{align*}
\item Back-off of a server of class $c \in \calC$ with at least one packet if the class activity vector allows the back-off, i.e., $\by\in \calY_c^-$:
\begin{align*}
& \quad {\begin{array}{lcll}
(-\frac{e_{c,i}}{N}, \by + e_c) & \text{at rate} & N \nu_c (x_{c,i} - x_{c,i+1}) & \text{for } 1 \leq i < B,\\
(-\frac{e_{c,i}}{N}, \by + e_c) & \text{at rate} & N \nu_c x_{c,i} & \text{for } i = B.\\
\end{array}}
\end{align*}
\item Transmission completion of an active node of class $c \in \calC$, i.e., $\by \in \calY^+_c$:
\begin{align*}
& \quad {\begin{array}{lcll} 
(0, \by - e_c) & \text{at rate} & N \mu_c. &  \\
\end{array}} 
\end{align*}
\end{itemize}
We denote the set of all possible transition vectors $(\ell/N,\by')$ from a state pair $(\bx,\by)$ by $\calT(\bx,\by)$.

\subsection{Steady-State distribution $\pi_y(x)$ and average drift}

By using the above transition definitions, the matrix $K$ is given by:
\begin{itemize}
    \item If $\by\in\calY^-_c$, then $K(\by,\by+e_c) = N\nu_c x_{c,1}$,
    \item If $\by\in\calY^+_c$, then $K(\by,\by-e_c) = N\mu_c$,
\end{itemize}
all other entries of the matrix being $0$.

This representation is used by the authors of \cite{cecchi_mean-field_2019} to derive the product-form stationary distribution for a fixed server state $\bx$. This product form is closely related to the product-form stationary distribution of saturated networks as found in \cite{van_de_ven_insensitivity_2010,wang_throughput_2005,boorstyn_throughput_1987}: The quantity $\pi_{\by}(x)$ is calculated as follows:
\begin{align*}
    \pi_{\by}(x) & := \frac{Z(\bx,\by)}{Z(\bx)}, & \text{ with }  && Z(\bx,\by) & = \prod_{c\in\calC} \bigl(\frac{\nu_c}{\mu_c}x_{c,1}\bigr)^{y_c}, & Z(\bx) & = \sum_{\by\in \calY} Z(\bx,\by).
\end{align*}

Following our definition of Section \ref{sec:def_drift} the drift and its average version are generically defined by:
\begin{align*}
F(\bx,\by) & = \sum_{(\ell, \by')\in\calT(\bx,\by)} \alpha_{\ell, \by'}(\bx,\by) \ell, & \text{and} && \barF(\bx) &  = \sum_{w\in\calY} \pi_{\by}(\bx) F(\bx,\by).
\end{align*}
For the drift $F(\bx,\by)$ of the random access model this leads to the closed form expression
\begin{align*}
F(\bx,\by) & = \sum_{c\in\calC}\bigl(\sum_{1< i \leq B} e_{c,i}\ \lambda_c(x_{c,i-1} - x_{c,i}) + e_{c,1}\ \lambda_c(1-x_{c,1}) \\
& \qquad \quad - \mathbf{1}_{\{y\in\calY_c^-\}}(\sum_{1\leq i < B} e_{c,i}\ \nu_c(x_{c,i} - x_{c,i+1}) + e_{c,B}\ \nu_c x_{c,B}) \bigr).
\end{align*}

It should be clear that assumption~\ref{assum:continuous} holds in our case because the rates given in Section~\ref{sec:example_rates} are all continuous in $\bx$ (in fact they are all linear).  Moreover, the model also satisfies Assumption~\ref{assum:irreducible}: For a given $\bx$, the set of irreducible states for $K(\bx)$ contains all the feasible activation vectors $y$ such that $y_c=0$ if $x_{(c,1)}=0$. The condition $x_{{c,1}}=0$ implies that the nodes of class $c$ do not have any packets to transmit. The situation of Assumption~\ref{assum:exp_stable} is more complicated. To the best of our knowledge, there has not been a complete stability characterization for the unsaturated random access CSMA model. Cecchi et al. show in \cite{cecchi_mean-field_2019} that in the case of a complete interference graph stability conditions can be derived which assure global exponential stability. They further conjectured that similar results hold for general interference graphs. In our analysis, we assume that \ref{assum:exp_stable} holds.

\subsection{Numerical results}

To study the accuracy of the mean field approximation and the refined term proposed in Theorem~\ref{thrm:steady-state}, we implemented a Python library\footnote{\url{https://gitlab.inria.fr/sallmeie/bias-and-refinement-of-multiscale-mean-field-models}} to simulate the system and compute the mean field approximation and the refinement term. This library is generic and can take as an input any instance of the model which we defined above. For instance, in the Code Cell~\ref{lst:python_code}, we illustrate how to use this library to construct a model where the interference graph is as in Figure~\ref{fig:5_node_graph}, the rates are $\lambda = [.5,.7,.7,.6,.4], \nu = [4,3,3,3,3],\mu = [3,3,2,4,2]$, and the buffer size is equal to $10$. This cell shows how to initialize the 5 node model and obtain the approximation and refinement from our implementation. We also perform the same experiments with the linear 3 node model, for which we provide the results in Appendix~\ref{apx:3_node_numerical_results}. Note that the results are qualitatively very similar.

\begin{python}[caption={Initialization and Computation of Mean Field and Refinements.},label={lst:python_code}]
# Graph structure (this is the five node example)
G = np.array([[0,1,1,1,0], 
            [1,0,0,0,0],
            [1,0,0,0,1],
            [1,0,0,0,1],
            [0,0,1,1,0]])  
# rates & buffer size
_lambda = np.array([.5,.7,.7,.6,.4])
nu = np.array([4,3,3,3,3])
mu = np.array([3,3,2,4,2])
buffer_size = 10

# We define the model, compute a trajectory and the refinement term. 
csma = symbolic_CSMA(nu, mu, _lambda, G, buffer_size)
T, X = csma.ode(time=200)                       # mean field ODE
v, s, t, w, u = csma.compute_refinements(X[-1]) # steady-state refinement
\end{python}

In order to compute the refinement terms, the library needs to compute various derivatives (of the drift or of the matrix $K^+(\bx)$). To implement this, we rely on symbolic differentiation provided by the \texttt{sympy} library \cite{meurer2017sympy}. As we see later in Table~\ref{table:computation_times}, the use of the symbolic differentiation is the performance bottleneck of our implementation.  In Appendix~\ref{apx:computational_notes} we furthermore show how to obtain the stationary distribution and the derivative of $K^+(x)$ numerically.

\paragraph*{Transient regime and illustration of Theorem~\ref{thrm:stoch_sys_ode}}
To illustrate the accuracy of the mean field, we use the 5 node model described in the Code Cell~\ref{lst:python_code}. We first simulate the CSMA model and compare it with the mean field ODE. The results are reported in Figure~\ref{fig:transient} where we plot the mean field approximation against a sample mean $\esp{\bX_t}$ derived from $1000$ simulations. Initially, all servers are idle. The plot shows the share of servers of class $c=3$ that have at least one job, that is $\esp{X_{t,(3,1)}}$. We compare the results for a model with $N=10$ servers per class (top right), $N=20$ (bottom right), or $N=50$ (left). We observe that in all cases, the evolution of the stochastic system is very well predicted by the mean field approximation. To quantify this more precisely, each plot contains a zoom on the trajectory between the time $t=8$ to $t=13$. These zooms show that for $N=50$, the quantity $\esp{X\toN_{t,(3,1)}}$ is almost indistinguishable from the mean field approximation. For $N=10$ or $N=20$, the estimated average is slightly above the mean field curve, but the confidence intervals remain almost equal to the error.

\begin{figure}[ht]
    \centering
    \includegraphics[
      height=5cm,
      keepaspectratio,
    ]{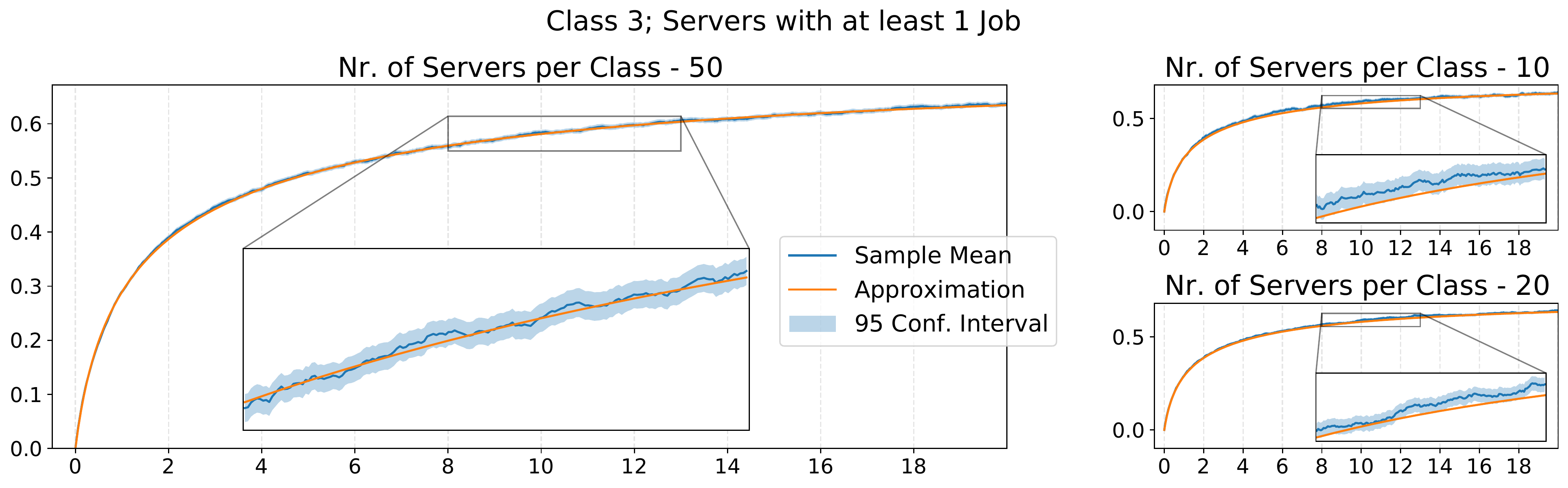}
    \caption{Illustration of the transient behavior of the CSMA model. We compare the `average' mean field and stochastic simulations for three different scaling parameters: $N=10$, $N=20$ and $N=50$.}
    \label{fig:transient}
\end{figure}

\paragraph*{Steady-state and refined accuracy}
While Theorems~\ref{thrm:stoch_sys_ode} and \ref{thrm:steady-state} provide a guarantee on the accuracy of the mean field approximation, Theorem~\ref{thrm:refinement} shows that it is possible to compute an approximation that is more accurate than the original mean field approximation.  We illustrate this in Figure~\ref{fig:avg_queue_len} where we show the steady-state average queue lengths for the same $5$ node graph. The sample mean and confidence interval are computed from $40$ steady-state samples which again are obtained from independent time-averages of $7.5\times 10^6$ events of the Markov chain after a warm-up of $2.5\times10^6$ events. 
For a class $c$ and a buffer size $b$, the quantity $\sum_{b=1}^B\esp{X_{(c,b)}}$ is equal to the steady-state average queue length of each server of class $c$.  In Figure~\ref{fig:avg_queue_len}, we consider different values of $N$, and calculate the average queue length by the following three methods:
\begin{itemize}
    \item By using a stochastic simulator of the original CSMA model.
    \item By using the fixed point of the mean field approximation: $\sum_{b=1}^B(\Phinf)_{(c,b)}$. 
    \item By computing the refinement term, $C$, of Theorem~\ref{thrm:refinement} and using $\sum_{b=1}^B(\Phinf)_{(c,b)}+(C)_{(c,b)}/N$.
\end{itemize}
When looking at the scale of the $y$-axis, we see that in all cases, the accuracy of the mean field approximation is already quite good. More importantly, we also observe that in all cases, the refined approximation seems almost exact: For all considered cases, the refined approximation lies within the 95 percent confidence interval of the simulations and seems to work well even for a small number of servers, $N \approx 10,20$. This result is similar to the one observed for one-timescale mean field models in \cite{gast_refined_2017}.
\begin{figure}[ht]
    \includegraphics[
      width=\textwidth,
      keepaspectratio,
    ]{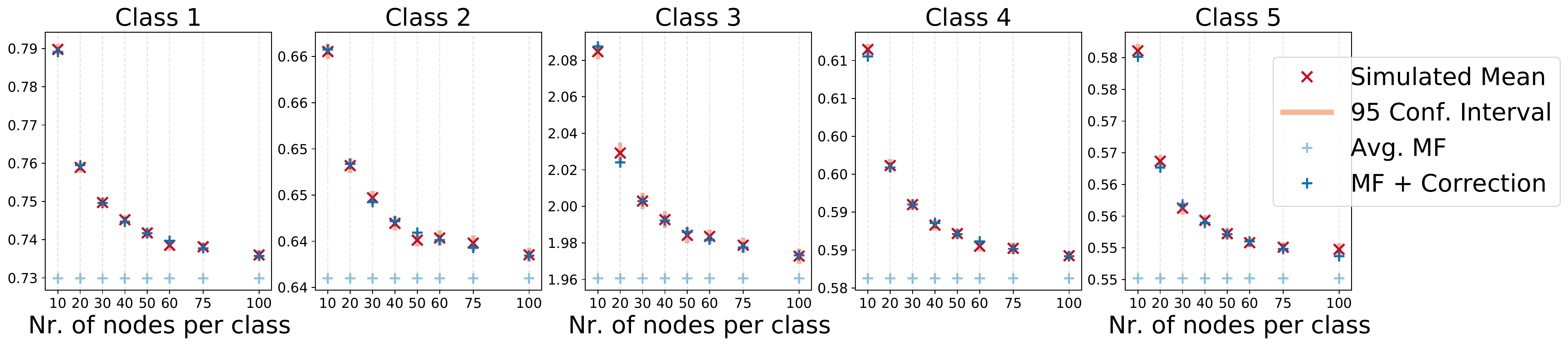}
    \caption{Queue Length Distribution for the 5 Node Graph of Figure~(\ref{fig:5_node_graph}).}
    \label{fig:avg_queue_len}
\end{figure}
\paragraph*{Computation Time}
While the previous figure shows that the refined approximation provides an increase in accuracy for small values of $N$, it comes at the cost of an increase in computation time because one needs to compute the various derivatives of the rate functions and to solve a new linear systems of equations.  In order to quantify the additional computation time, we measure the time taken by our implementation to compute the refinement terms which are reported in Table~\ref{table:computation_times}. We compare the $5$ node model studied before and a $3$-node model whose interference graph is as in Figure~\ref{fig:linear_graph}.  We observe that the time taken to compute the refinement term is significant, in particular for the $5$-node model.  Yet, when looking more carefully at what takes time,  we realize that most of the computation time is taken by the symbolic differentiation. Indeed, to simplify our implementation, we used the automated differentiation method of \texttt{sympy}. While this yields simplifications for the implementation, we encountered that it massively slows down the refinement computation times. Through code profiling it showed that around 95 percent of the computing time is taken by \texttt{sympy} methods such as differentiation and evaluation of symbolic expressions. For smaller interference graphs, e.g., linear 2 / 3 node graphs, this effect is not limiting. For larger graphs, the differentiation turns out to be the restricting factor. In Table~\ref{table:computation_times} we state the computation times for a linear 3 node model and for the setup described before.

We would like to emphasize that the goal of our implementation is to illustrate the theoretical statements, and thus we did not focus on efficiency. The table shows that if one wants to adapt our implementation to work with larger graphs, it would be sufficient to implement a more efficient differentiation method. For instance, this could be done by using closed form expression of the derivatives, or by using automatic differentiation methods, or by using finite difference methods. We believe that such methods would probably be much faster. 

\begin{table}[ht]
  \begin{tabular}{l|c|c|c|c|c|c}
     & Jakobian (A) & Hessian (B) & v + w & s + t + u & Total & Sympy\\
    \hline
    3 Node & 6.08 (6.08) & 30.71 (30.71) & 0.19 (0) & 5.64 (3.81) & 42.63 (40.6) & 95.22\% \\
    5 Node & 97.09 (97.09) & 897.15 (897.15) & 1. (0) & 122.36 (86.09) & 1117.6 (1080.33) & 96.66\% \\
  \end{tabular}\\
  \caption{Refinement computation times for Random Access CSMA model for a 3 node linear graph and the 5 node graph as in Figures \ref{fig:linear_graph}, \ref{fig:5_node_graph}.
  Times are given in seconds. In parentheses, we indicate the time taken by differentiation and subsequent sympy methods. These numbers show that the Sympy code takes more than $95\%$ of the computation time.  }
  \label{table:computation_times}
\end{table}

\section{Conclusion}

In this paper we investigate the accuracy of the classical \emph{averaging method} that is used to study two timescale models.  We study a generic two timescale model and show that under mild regularity conditions, the bias of this `average' mean field approximation is of order $O(1/N)$. This result holds for any finite time-horizon, and extends to the steady-state regime under the classical assumption that the system has a unique and stable fixed point. Our results show the existence of a bias term $C_h(t)$ for any regular function $h$: 
\begin{align*}
    \E[h(\bX_t)] = \underbrace{\underbrace{h(\Phi_t(\bx))}_{\text{classical `average' mean field}} + \underbrace{\frac1N C_h(t)}_{\text{$O(1/N)$ expansion of the bias}}}_{\text{refined `average' mean field}} + \quad o(1/N).
\end{align*}
For the steady-state regime $t=+\infty$, we propose an algorithmic method to calculate this term $C_h$.  This correction term can be computed by solving linear systems and is therefore easily numerically computable. We show on an example that, similarly to what was done for classical one timescale models \cite{gast_refined_2017}, the bias term leads to an approximation that is almost exact for small values of $N$ like $N=10,20$.

An interesting open question would be to obtain a characterization of $C_h(t)$ for the transient regime. Yet, it is not clear to us if those expressions would be usable as their size grows quickly with the system size. From an application point of view, our examples show that the new approximation leads to very accurate estimates for CSMA models. We believe that the same should hold for other multiscale models.


\begin{acks}
  This work is supported by the French National Research Agency (ANR) through REFINO Project under Grant ANR-19-CE23-0015.  
\end{acks}


\bibliographystyle{ACM-Reference-Format}
\bibliography{Refs}


\appendix

\section{Definitions}

In this section we revise some essential definitions and properties used in the paper. As these definitions are well established, we only briefly recall them to provide a self-contained paper.

\subsection{\texorpdfstring{$C_0$}{C0}-Semi-Group, Hölder Norm, ODE differentiability}

\begin{definition}[$C_0$ Semi-Group \cite{pazy_semigroups_1983} Definition 2.1 ] \label{def:semi-group}
$T_s$ is called strongly continuous semi-group (or $C_0$-semi-group) if 
\begin{align}
& T_0 = Id, 
&& T_{s+t} = T_s T_t \text{ for all } s,t \geq 0,
&& \lim_{t\downarrow 0} T_t z = z \text{ for all } z.
\end{align}
\end{definition}
Since we will only work with $C_0$-semi-groups we will simply refer to them as semi-groups. \\
\noindent\textbf{Semi-group \& Generator commutation} \label{appdx:semi_group_commute} The generator of a $C_0$-semi-group is defined by
\begin{align*}
A z = \lim_{t \to 0}\frac1t(T_t z - T_0 z).
\end{align*}
A direct consequence of the definition of the generator and a standard property is that it commutes with its defining $C_0$-semi-group, i.e.,
\begin{align*}
AT_sh(z) = T_sAh(z) = \frac{d}{ds} T_sh(z).
\end{align*}
Note, this therefore holds true for the semi-groups given by the stochastic system $\Psi_s  h (x,y) = \E[h(X_s,Y_s)\mid X_0,Y_0 = x,y]$ with generator $Lh(x,y)$ as in Equation~\eqref{eq:inf_generator_eq} and the semi-group of the ODE $\Phi_s h (x) = h(\phi_s(x))$ with generator $\Lambda h (x) = D_x h(x) \barF(x)$.


\begin{definition}[Hölder Norm and Space]
\label{def:holder_norm}
For $U\subset \R^n$ and $u\in C^k(U)$ 
\begin{align*}
\norm{u}_{k,\gamma} := \sum_k \sup_{x\in U} \norm{ D^k u(x)} + \sum_k \sup \left\{ \frac{\norm{D^k u(x) - D^ku(y)}}{\norm{x-y}^\gamma} \mid x,y \in U, x\neq y \right\}
\end{align*}
is called Hölder norm. The space of functions for which the norm is finite is called Hölder space and denoted by $\calD^k_\gamma(U)$. For the case $\gamma = 1$ the Hölder space encloses all functions who are $k$-times continuously differentiable with bounded derivatives and who's $k$-th derivatives are Lipschitz continuous. The latter we simply denote by $\calD^k(U)$.  
\end{definition}
\noindent An important implication is that all Hölder continuous functions are uniformly continuous. 


\begin{lemma}[Drift induced differentiability \cite{perko_differential_2001} Theorem 1 p.80]
\label{lem:ode_smooth}
    Let $E$ be an open subset of $\R^n$ containing $x_0$ and assume that $f\in C^k(E)$. Then there exists an $a>0$ and $\delta>0$ such that for all $y \in N_\delta(x_0)$\footnote{$N_\delta(x_0):= \{x\in\R^n : \|x - x_0\| \leq \delta\}$} the initial value problem
    \begin{align*}
       \dot{x} = F(x), \qquad x(0) = y,
    \end{align*}
    has a unique solution $u$ which is $k$-times continuously differentiable with respect to the initial condition for $t\in[-a,a]$. 
\end{lemma}

\begin{proof}
    Theorem 1 p.80-83 and Remark 1 p.83 of \cite{perko_differential_2001}.
\end{proof}

\section{Technical Lemmas and Proofs}

\subsection{Proof of Lemma~\ref{lemma:g-inverse_poisson-solution}}
\label{apx:proof_poisson}

By assumption~\ref{assum:irreducible}, the transition matrix $K(x)$ has a unique irreducible class. As pointed out in Section~\ref{sec:assumptions}, the corresponding Markov chain has a unique stationary distribution that we denote by $\pi(x)$. Let $\Pi(x) := \mathds{1}\pi^T(x)$ be the matrix whose lines are all equal to $\pi(x)$.  By~\cite{hunter_generalized_1982}(Theorem 3.5, p.17), $\bigl(K(x) + \Pi(x)\bigr)$ is non-singular and its inverse $(K(x)+\Pi(x))^{-1}$ is a generalized inverse to $K(x)$, which means that it satisfies $K(x)(K(x)+\Pi(x))^{-1} K(x)= K(x)$.

To obtain the solution to the Poisson equation \eqref{eq:PoissonY} we only consider the case where $h$ takes values in $\R$. The extension to a function $h$ that takes values in $\R^n$ is straightforward as it corresponds to $n$ independent Poisson equations.

Let us suppress the dependence on $x$ for clarity. Recall that $K^+=(K+\Pi)^{-1}(I-\Pi)$ and let us study the product $KK^+$: 
\begin{align*}
    K K^+ &= K(K+\Pi)^{-1} (I-\Pi) & \text{\it (by definition)}\\
    &= K(K+\Pi)^{-1} (I+K-K-\Pi)\\
    &= K(K+\Pi)^{-1} + K(K+\Pi)^{-1}K - K(K+\Pi)^{-1}(K+\Pi)& \text{\it (expanding the product)}\\
    &= K(K+\Pi)^{-1} & \text{\it (the last two terms equal $\pm K$)}\\
    &= (K+\Pi)(K+\Pi)^{-1} - \Pi(K+\Pi)^{-1} & \text{\it (Adding and subtracting $\Pi(K+\Pi)^{-1}$)}\\
    &= I - \Pi,
\end{align*}
where the last equality holds because $\pi^T\mathds{1}=1$ and therefore $\Pi \ \Pi=\mathds{1}\pi^T\mathds{1}\pi^T=\mathds{1}\pi^T=\Pi$. Combined with $\Pi \ K=0$, this shows that $\Pi = \Pi(K+\Pi)$ and therefore $\Pi(K+\Pi)^{-1}=\Pi$.

The above computations show that if $\Gfast_h(x,y) = \sum_{y'}K^+_{y,y'}(x) h(x,y')$, then:
\begin{align*}
    K(x)\Gfast_h(x,y) &= h(x,y) - \sum_{y'} \pi_{y'}(x) h(x,y'),
\end{align*}
which shows that $\Gfast_h$ is the solution of the Poisson equation.

The differentiability of $\Gfast_h$ follows from the differentiability of $h$ and $K^+$: Under Assumption~\ref{assum:continuous} $K(x)$ is continuously differentiable. By Assumption~\ref{assum:irreducible}, $K(x)$ has a unique irreducible class, this implies $\Pi(x)$ is continuously differentiable, for which we refer to \cite{ipsen_uniform_1994}, and therefore further implies that $(I+\Pi(x))^{-1}$ and $K^+(x)$ are continuously differentiable. This proves that $\Gfast_h$ is continuously differentiable in $x$.

\subsection{Technical Lemmas used to prove Theorem~\ref{thrm:stoch_sys_ode} (Transient Regime)}
\label{apx:lemma_transient}

\begin{lemma}\label{lemma:h_s}
    For arbitrary but fixed $t>0$, let $g:(s,x,y)\in[0,t]\times\calX\times\calY \mapsto g_s(x,y) \in \R$ be a continuous function that is continuously differentiable in $s$ and let $H_s(x,y):=\esp{g_{s}(X_s,Y_s)\mid X_0,Y_0=x,y}$. Then:
    \begin{align*}
        H_t(x,y) - H_0(x,y) = \int_0^t \esp{\frac{d}{d\tau} g_{s+\tau}(X_s,Y_s)\mid_{\tau=0} } ds + \int_0^t \esp{ L g_s(X_s,Y_s)}ds.
    \end{align*}
\end{lemma}


\begin{proof}
    By definition of the generator $L$, the quantity $H_t(x,y)=\esp{g_{t}(X_t,Y_t)\mid X_0,Y_0=x,y}$ is right sided differentiable, i.e., $\frac{d^+}{ds} f(s)= \lim_{ds\to 0^+} \frac1{ds} \bigl(f(s+ds) - f(s)\bigr)$, with respect to time. Using semi-group properties and bounds, its derivative is 
    \begin{align*}
        \frac{d^+}{ds}H_s(x,y) &= \esp{\frac{d}{d \tau} g_{s+\tau}(X_s,Y_s)\mid_{\tau=0} } + \esp{ L g_s(X_s,Y_s)}.
    \end{align*}
    The first term corresponds to the derivation of $g_{s}$ with respect to time and the second term to the changes of the stochastic system in $(X_s,Y_s)$. The lemma therefore follows by using that $H_t(x,y) - H_0(x,y) = \int_0^t \frac{d^+}{ds} H_s(X_s,Y_s)$.
\end{proof}


\begin{lemma}[Bound for $\Lf g(x,y)$] \label{lemma:lfast_bound}
  Take the two-timescale stochastic system as introduced in Section \ref{sec:def_stoch_sys} with generator $L$ and assume \ref{assum:continuous} and \ref{assum:irreducible}. Let $\Lf$ be as defined in Section \ref{sec:fast_poisson_regularity}. Then, for $g\in \calD^1(\calX \times \calY)$ and $(x,y) \in \calX \times \calY$
  \begin{align*}
  \Lf g(x,y) - \frac1NLg(x,y)  = \frac1N\sum_{\ell, y'}\alpha_{\ell,y'}(x,y) \ell \ D_x g(x, y') + o(1/N).
  \end{align*}
\end{lemma}
    

\begin{proof}
By definition of $\Lf$, for a continuous function $g$ the values of $\Lf g(x,y)$ and $\frac{1}{N}Lg(x,y)$ coincide in the limit. For finite $N$, we first look at $\frac1N Lg(x,y)$ which is given by
\begin{align*}
\frac1N Lg(x,y) = \frac1N \sum_{\ell, y'} N \alpha_{\ell,y'}(x,y)( g(x + \frac{\ell}{N}, y') - g(x,y)).
\end{align*}
Using the continuity of $g$ in $x$, and definition of $\Lf g(x,y)$ we have 
\begin{align*}
\Lf g(x,y) = \lim_{N\to \infty}\frac1N Lg(x,y) &= \sum_{\ell, y'} \alpha_{\ell,y'}(x,y)( \lim_{N\to \infty} g(x + \frac{\ell}{N}, y') - g(x,y)) \\ 
& = \sum_{\ell, y'} \alpha_{\ell,y'}(x,y)(g(x, y') - g(x,y)).
\end{align*}
Using Taylor's theorem
\begin{align*}
Lg(x,y) &= N \sum_{\ell, y'} \alpha_{\ell,y'}(x,y)( g(x + \frac{\ell}{N}, y') - g(x,y)) \\
& = N \sum_{\ell, y'} \alpha_{\ell,y'}(x,y)( g(x, y') - g(x,y) + \ell \ D_x g(x, y')  + o(\norm{\ell})) \\
& = N \Lf g(x,y) + \sum_{\ell, y'}\alpha_{\ell,y'}(x,y) \ell \ D_x g(x, y')  + o(1).
\end{align*}

\end{proof}
    

\subsection{Technical Lemma used to prove Theorem~\ref{thrm:refinement} (steady-state refinement)}
\label{apx:proof_steady-state}


The next Lemma justifies that the first term of \eqref{eq:proof_ref_sep_error_terms} is approximated in the limit by the `average' version of $J_h$ as defined in Section~\ref{proof:refinement}. Second, Lemma~\ref{lemma:deriv_of_G_slow} gives stability conditions for the solution of the \ref{eq:Poisson:B}.
\begin{lemma}
\label{lemma:refinement_2}
\raggedright
 Assume \ref{assum:continuous}-\ref{assum:exp_stable} in particular assume $h\in\calD^1(\calX)$. Further assume that $\Gfast_F$, solution to the \ref{eq:Poisson:A}, is cont. differentiable in $x$ and $\Gslow_h$, solution to the \ref{eq:Poisson:B}, is twice cont. differentiable in $x$. Let $\bar{J}(\Phinf) = \sum_y \pi(\Phinf) J(\Phinf,y)$ be the `average' version of $J_h(x,y)=\sum_{\ell, y'} \alpha_{\ell,y'}(x,y) \  D_x \bigl(D_x\Gslow_h(x) \Gfast_F(x,y') \bigr) \ell$ in $\Phinf$, then 
\begin{align*}
& N \E [ D_x \Gslow_h(\Xinf)\bigl(\bar{F}(\Xinf) - F(\Xinf, \Yinf)\bigr) ] = \bar{J}(\Phinf) + o(1).
\end{align*}
\end{lemma}


\begin{proof}
To prove the lemma, let $\Gfast_F$ be the solution to the Poisson equation as given in \eqref{eq:Poisson:A}. We use $\Gfast_F$ to rewrite
\begin{align*}
N & \E[ D_x \Gslow_h(\Xinf) ( \bar{F}(\Xinf) - F(\Xinf, \Yinf) )] = - N \E[ D_x \Gslow_h(\Xinf) (F(\Xinf, \Yinf) - \bar{F}(\Xinf))]\\
& = - N \E[ D_x \Gslow_h(\Xinf) ( \Lf \Gfast_F(\Xinf, \Yinf))] 
 = - N \E[ \Lf D_x \Gslow_h(\Xinf) (  \Gfast_F(\Xinf, \Yinf))]. 
\end{align*}
Adding $\E[L D_x \Gslow(\Xinf) (  \Gfast_h(\Xinf, \Yinf))] = 0$ and applying the steps as in the proof of Lemma \ref{lemma:lfast_bound} we see that 
\begin{align}
& - N \E[ \Lf D_x \Gslow_h(\Xinf) (  \Gfast_F(\Xinf, \Yinf)) - L D_x \Gslow_h(\Xinf) (  \Gfast_F(\Xinf, \Yinf))] \nonumber \\
& = - \E [ \sum_{\ell,y'} \alpha_{\ell, y'}(\Xinf, \Yinf)  \bigl( D_x (D_x \Gslow_h(\Xinf)\Gfast_F(\Xinf, y')  ) \bigr) \ell \ ] + o(1). \label{eq:J-term-ref}
\end{align}
The last equality follows directly from the definition of $\Lf$ since $G_h$ only depends on $x$.
Using Theorem \ref{thrm:steady-state} we see that $\E[J_h(\Xinf, \Yinf)]$ with $J_h(x,y):=-\sum_{\ell, y'} \alpha_{\ell,y'}(x,y)\ell \  D_x \bigl(D_x\Gslow_h(x) \Gfast(x,y') \bigr)$ of equation \eqref{eq:J-term-ref} is approximated by $\bar{J}(\Phinf) = \sum_y \pi(\Phinf) J(\Phinf,y)$. This concludes the proof as it implies

$N \E [ D_x \Gslow_h(\Xinf)\bigl(\bar{F}(\Xinf) - F(\Xinf, \Yinf)\bigr) ] 
= \E [ J_h(\Xinf, \Yinf)] + o(1) = \bar{J}_h(\Phinf) + o(1).$
\end{proof}

\subsection{Stability of \texorpdfstring{$\Gslow_h$}{the `slow' Poisson equation}}

\begin{lemma}[Stability] \label{lemma:deriv_of_G_slow}
Assume that $\barF$ and $\phi$ are k-times differentiable with uniformly continuous derivatives and that $\phi$ has a unique exponentially stable attractor $\Phinf$. Then the k-th derivative of $\Gslow_h: x \mapsto \int_0^\infty h(\phi_s x) - h(\Phinf)ds$ is bounded and equal to $\int_0^\infty D^k_x (h \circ\phi_s )(x) ds$.
\end{lemma}

\begin{proof}
    This is a consequence of \cite{gast_refined_2017}[Lemma 3.5].
\end{proof}

\subsection{Proof of Proposition~\ref{prop:closed_form_J}}
\label{apx:proof_closed_form_J}

\begin{proof}
\label{proof:closed_form_new}
In the first part of the proof we find computable expressions of $D_x \Gslow_h, D_x^2 \Gslow_h $ and $D_x \Gfast_F$. These expressions allow us to rewrite $J_h$ and construct the closed form representation of $J_h$ as well as $\bar{J}_h$ in the following steps.
By definition
\begin{align}
J_h(x,y) & = -\sum_{\ell, y'} \alpha_{\ell,y'}(x,y) \  D_x \bigl(D_x\Gslow_h(x) \Gfast_F(x,y') \bigr) \ell \notag \\
& = - \sum_{\ell, y'} \alpha_{\ell,y'}(x,y)\  \bigl( \Gfast_F(x,y')^T D^2_x\Gslow_h(x)  + D_x \Gslow_h(x) D_x \Gfast_F(x,y')\bigr) \ell. \label{eq:J_equation_closed_form_lemma}
\end{align}

By Lemma \ref{lemma:g-inverse_poisson-solution} there exists a matrix $K^+(x)$ such that $\Gfast_F(x,y')$ has the form 
\begin{align*}
\Gfast_F(x,y') = \sum_{y'\in \calY} K^+_{y',y''}(x) F(x,y'').
\end{align*}
Assumption \ref{assum:continuous} which assures differentiability of the transition rates $\alpha$ with respect to $x$, also implies differentiability for $\Gfast_F$. Therefore, $D_x \Gfast_F(x,y') = \sum_{y''\in \calY} F(x,y'') \nabla^T_x K^+_{y',y''}(x) + K^+_{y',y''}(x) D_x F(x,y'')$ with $\nabla_x^T = \left[ \frac{\partial}{\partial x_1},\dots, \frac{\partial}{\partial x_n} \right], D_x f = (\frac{\partial}{\partial x_j} f_i)_{i,j=1\dots n}$. Using the results of \cite{gast_refined_2017}[Lemma 3.6] with $A, B$ the first and second derivative of $\bar{F}$ as defined in \eqref{eq:refinement_A_B_def}, it holds that for the equilibrium point $\Phinf$
\begin{align*}
D_x \Gslow_h(\Phinf)_i & = \sum_j \frac{\partial h}{\partial x_j} (\Phinf) \int_0^\infty \bigl(D_x \phi_s(\Phinf)\bigr)_{j,i}ds = \sum_j \frac{\partial h}{\partial x_j} (\Phinf) (-A)^{-1}_{j,i}
\end{align*}
as well as
\begin{align}
D^2_x \Gslow_h(\Phinf)_{n,m} & = \sum_{i,j} \frac{\partial^2 h}{\partial x_i \partial x_j} (\Phinf) \int_0^\infty \bigl(D_x \phi_s(\Phinf)\bigr)_{j,n} \bigl(D_x \phi_s(\Phinf)\bigr)_{i,m} ds \label{eq:decond_deriv_slow_poisson} \\
& \qquad + \sum_i \frac{\partial h}{\partial x_i} (\Phinf) \int_0^\infty \bigl( D_x^2 \phi_s(\Phinf) \bigr)_{j,m,n}ds \notag \\
& = \sum_{i,j} \frac{\partial^2 h}{\partial x_i \partial x_j} (\Phinf) \int_0^\infty \bigl(e^{As}\bigr)_{j,n} \bigl(e^{As}\bigr)_{i,m} ds \notag \\
& \qquad + \sum_i \frac{\partial h}{\partial x_i} (\Phinf) \bigl( \sum_{j} (-A)^{-1}_{i,j} \sum_{k_1,k_2} B_{j,k_1,k_2} \int_0^\infty (e^{As}\bigr)_{k_1,n} (e^{As}\bigr)_{k_2,m} ds \bigr). \notag
\end{align}
Next, the above equations are used to rewrite \eqref{eq:J_equation_closed_form_lemma}. To obtain a closed form expression for the left summand, we start by looking at the sum
\begin{align}
\sum_{y}\pi_y(\Phinf)\sum_{\ell, y'} \alpha_{\ell,y'}(\Phinf, y)  \sum_{y'}K^+_{y',y'}F_m(\Phinf,y')\ell_n\int_0^\infty (e^{As}\bigr)_{k_1,n} (e^{As}\bigr)_{k_2,m}. \label{eq:lyapunov_equal}
\end{align}
A solution to the above equation is given by the following Lyapunov equation. To ease notations, define
\begin{align*}
O_{m,n} & := \sum_{y}\pi_y(\Phinf) \sum_{\ell, y'} \alpha_{\ell,y'}(\Phinf,y) K^+_{y',:}(\Phinf)F_m(\Phinf,:) \ell_n,
\end{align*}
or, equivalently in matrix notation,
$O \phantom{:}= \sum_{y}\pi_y(\Phinf)\sum_{\ell, y'} \alpha_{\ell,y'}(\Phinf,y) K^+_{y',:}(\Phinf)F(\Phinf,:) \ell^T.$

If a matrix $U$ solves\footnote{As $A$ is non-singular, $A$ and $-A^T$ don't share any eigenvalues and therefore equation \eqref{eq:sylvester_equation_J_bar} has a unique solution.} the Sylvester equation (for $X$)
\begin{align}
A X + X A^T = -O, \label{eq:sylvester_equation_J_bar}
\end{align}
it is equal to \eqref{eq:lyapunov_equal}. Applying this identity and \eqref{eq:decond_deriv_slow_poisson} to the first summand of $\bar{J}(\Phinf)$ which is the `average' version of \eqref{eq:J_equation_closed_form_lemma}, lets us rewrite
\begin{align*}
& \sum_{y} \pi_y(\Phinf) \sum_{\ell, y'} \alpha_{\ell,y'}(\Phinf,y)\  \Gfast_F(\Phinf,y')^T D^2_x\Gslow_h(\Phinf) \ell \\ 
& \qquad = \sum_{i,j} \frac{\partial^2 h}{\partial x_i \partial x_j}(\Phinf) U_{i,j} + \sum_i \frac{\partial h}{\partial x_i} (\Phinf) \sum_{j} (-A)^{-1}_{i,j} \sum_{k_1,k_2} B_{j,k_1,k_2} U_{k_1,k_2}.
\end{align*}
For the second major summand appearing in the definition of $\bar{J}(\Phinf)$, writing out the solutions to the Poisson equations and their derivatives yields 
\begin{align*}
&  \sum_{y} \pi_y(\Phinf) \sum_{\ell, y'} \alpha_{\ell,y'}(\Phinf,y)\  D_x\Gslow_h(\Phinf) D_x \Gfast_F(\Phinf,y') \ell \\
& = \sum_{y} \pi_y(\Phinf) \sum_{\ell, y'} \alpha_{\ell,y'}(\Phinf,y) \times \\
& \qquad \sum_{j,k} \left( \sum_i  \frac{\partial h}{\partial x_i}(\Phinf) (-A)^{-1}_{i,k} \right) \left( \sum_{y'} \frac{\partial}{\partial x_j} K_{y',y'}^+(\Phinf) F_k(\Phinf,y') + K_{y',y'}^+(\Phinf) \frac{\partial}{\partial x_j} F_k(\Phinf,y') \right) \ell_j.
\end{align*}
By using vector notation and rearranging the sums this is equal to
\begin{align}
& \sum_{i} \frac{\partial h}{\partial x_i}(\Phinf) \sum_k (-A)^{-1}_{i,k} \sum_{y} \pi_y(\Phinf) \sum_{\ell, y'} \alpha_{\ell,y'}(\Phinf,y) \times \nonumber \\
& \qquad \left( \sum_{y'} F_k (\Phinf,y') \nabla^T_x K_{y',y'}^{+}(\Phinf) \ell + K_{y',y'}^{+}(\Phinf) \nabla^T_x F_k(\Phinf,y') \ell  \right). \label{eq:def_s}
\end{align}
Lastly, define  $T_i := \sum_{j} A^{-1}_{i,j} \sum_{k_1,k_2} B_{j,k_1,k_2} U_{k_1,k_2}$ and
\begin{align*}
S_i := \sum_k A^{-1}_{i,k} \sum_{y} \pi_y(\Phinf) \sum_{\ell, y'} \alpha_{\ell,y'}(\Phinf,y) \left( \sum_{y'} F_k (\Phinf,y') \nabla^T_x K_{y',y'}^{+}(\Phinf) \ell + K_{y',y'}^{+}(\Phinf) \nabla^T_x F_k(\Phinf,y') \ell  \right).
\end{align*} This concludes the proof as by definition of $T$,$S$ and $U$,   $\bar{J}(\Phinf) = \sum_i \frac{\partial h}{\partial x_i}(\Phinf) (T_i + S_i) + \sum_{i,j} \frac{\partial^2 h}{\partial x_i \partial x_j}(\Phinf) U_{i,j}$.
\end{proof}

\section{Computational Notes}
\label{apx:computational_notes}

The first note describes how to calculate the steady-state probabilities $\pi(x)$ associated to the transition matrix $K(x)$.

\begin{note}[Note on the computation for the stationary probabilities.]
As the finite state continuous time Markov chain with generator $K(x)$ has unique irreducible class, there exists a non-trivial unique stationary distribution $\pi(x)$\footnote{As the Markov chain is allowed to have transient states the stationary distribution can take zero values.}. Denote by $\{1,\dots, m\}$ the states of the Markov chain. $\pi(x)$ is obtained by solving the linear system
    \begin{align*}
        v \ K(x) = \mathbf{0}, \quad \sum_{i=1}^m v_i = 1,
    \end{align*}
    with $v\geq0$ component wise. By definition, $K(x)$ is of rank $m-1$. Using its structure, i.e., $\sum_{y'} K(x)_{y,y'} = 0 \ \forall y$, we can rewrite the above over-determined linear system by replacing the last column of the generator with $\mathds{1} = [1,\dots,1]^T$ yielding
    \begin{align*}
    v \ [ K(x)_{:,1}, \dots, K(x)_{:,m-1}, \mathds{1}] = [\underbrace{0,\dots,0}_{m-1 \text{ times}}, 1],
    \end{align*}
    for which the solution is the stationary distribution $\pi(x)$. By $K(x)_{:,y}, \ y=1,\dots,m-1$ we denote the $y$-th column of $K(x)$.
\end{note}

As shows in Lemma~\ref{lemma:g-inverse_poisson-solution} the solution to the 'fast' Poisson Equation~\ref{eq:PoissonY} has the form $\Gfast_h(x,y) = \sum_{y'}K^{+}_{y,y'}(x)h(x,y')$. To compute the bias correction terms, it is necessary to calculate the first derivative of $\Gfast_h$ with respect to $x$. Since the computation of the derivative of $(K(x) + \Pi(x))^{-1}$ can be non-trivial and time consuming, the note below elaborates how the derivative can be efficiently obtained.
\begin{note}[Computation of \texorpdfstring{$D_x K^+(x)$}{the Poisson Solution derivative}]
By definition $K^+(x) = (K(x) + \Pi(x))^{-1} (I - \Pi(x))$. Using basic matrix derivation rules one has
\begin{align}
 \frac{\partial}{\partial x_i} K^+(x) & =  \frac{\partial}{\partial x_i} \bigl( (K(x) + \Pi(x))^{-1} (I - \Pi(x)) \bigr) \nonumber \\
& =  \frac{\partial}{\partial x_i} \bigl(K(x) + \Pi(x)\bigr)^{-1}  \bigl(I - \Pi(x)\bigr) - \bigl(K(x) + \Pi(x)\bigr)^{-1}  \frac{\partial}{\partial x_i} \Pi(x). \label{eq:partial_deriv_K+}
\end{align}
Since the numerical difficulty lies only the computation of $D_x\bigl(K(x) + \Pi(x)\bigr)^{-1}$ we focus solely on it. Define $E(x) = (K(x) + \Pi(x))$ and let $I$ be the identity matrix. As pointed out in Lemma~\ref{lemma:g-inverse_poisson-solution} $(K(x) + \Pi(x))$ is indeed invertible and thus for the partials derivative $\frac{\partial}{\partial x_i}, i=0\dots d_x$ the following holds:
\begin{align*}
&& \frac{\partial}{\partial x_i} I &= \frac{\partial}{\partial x_i} ( E(x) E^{-1}(x) ) \\ 
\Leftrightarrow && 0 & = (\frac{\partial}{\partial x_i} E(x)) E^{-1}(x) + E(x) \frac{\partial}{\partial x_i} E^{-1}(x) \\
\Leftrightarrow && \frac{\partial}{\partial x_i} E^{-1}(x) & = - E^{-1}(x) (\frac{\partial}{\partial x_i} E(x)) E^{-1}(x) \\
\Leftrightarrow && \frac{\partial}{\partial x_i} (K(x) + \Pi(x))^{-1} & =  - (K(x) + \Pi(x))^{-1} \bigl(\frac{\partial}{\partial x_i} (K(x) + \Pi(x))\bigr) (K(x) + \Pi(x))^{-1}.
\end{align*}
The above can now be used to compute Equation~\eqref{eq:partial_deriv_K+}.
\end{note}


\section{Numerical Results for the 3 Node Model}
\label{apx:3_node_numerical_results}

For completeness, we give the numerical results of the linear 3 node model of Graph~\ref{fig:linear_graph}. To obtain the results we modify the parameters of the code of Code Cell~\ref{lst:python_code} to match with the 3 node setup. The new model is defined as in Code Cell~\ref{lst:python_code_3_node}.

\begin{python}[caption={Initialization of Mean Field for the 3 Node Model.},label={lst:python_code_3_node}]
# Graph structure (this is the 3 node example)
G = np.array([[0,1,0],
            [1,0,1],
            [0,1,0]])
# rates & buffer size
_lambda = np.array([0.4,0.2,0.5])
nu = np.array([1.2,2.,1.5])
mu = np.array([1.4,1.3,1.7])
buffer_size = 10
\end{python}

As seen in Figure~\ref{fig:3_node_avg_queue_len}, we obtain similar results as for the 5 node model when considering steady-state average queue length values. The sample mean and confidence interval are computed from $40$ steady-state samples in the same manner as for the 5 node example. The refined approximation almost exactly indicates the stationary value of the stochastic process even for small $N$ while the mean field approximation gets more accurate as $N$ grows. 

\begin{figure}[ht]
    \includegraphics[
      width=\textwidth,
      keepaspectratio,
    ]{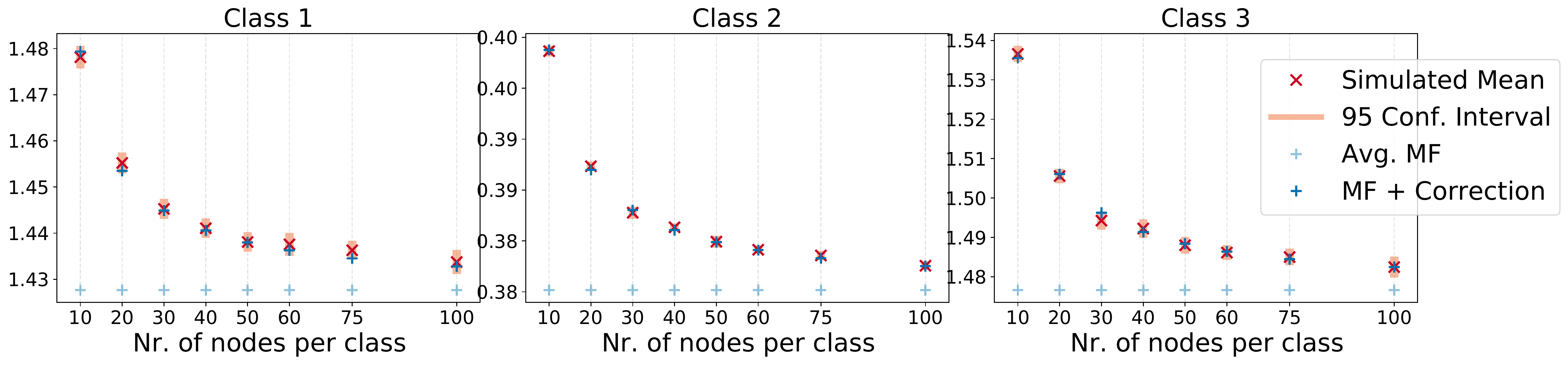}
    \caption{Stationary Queue Length Distribution for the 3 Node Graph of Figure~(\ref{fig:linear_graph}).}
    \label{fig:3_node_avg_queue_len}
\end{figure}
    
\end{document}